\documentclass[12pt]{amsart}
\usepackage{amsmath, amssymb, enumitem, float,  mathrsfs,amsaddr,graphicx}
\usepackage{hyperref}
\usepackage[headings]{fullpage}

\numberwithin{equation}{section}

\theoremstyle{plain}

\newtheorem{theorem}[subsubsection]{Theorem}
\newtheorem{lemma}[subsubsection]{Lemma}
\newtheorem{proposition}[subsubsection]{{Proposition}}
\newtheorem{corollary}[subsubsection]{{Corollary}}

\theoremstyle{definition}
\newtheorem{definition}[subsubsection]{{Definition}}
\newtheorem{example}[subsubsection]{{Example}}

\theoremstyle{remark}
\newtheorem{remark}[subsubsection]{{Remark}}

\newtheoremstyle{RepTheorem} 
	{\topsep}{\topsep}
	{\itshape}
	{}
	{\bfseries}
	{.}
	{ }
	{\thmname{#1}\thmnote{ \bfseries #3}}
\theoremstyle{RepTheorem}
\newtheorem{reptheorem}[subsubsection]{Theorem}

\def\Z {{\mathbb Z}}
\def\F {{\mathbb F}}
\def\Q {{\mathbb Q}}
\def\R {{\mathbb R}}
\def\C {{\mathbb C}}
\def\A {{\mathbb A}}

\def\P {{\mathbb P}}
\def\G {{\mathbb G}}

\def\O {{\mathcal O}}

\def\mcB {{\mathcal B}}
\def\mcD {{\mathcal D}}
\def\mcE {{\mathcal E}}
\def\mcF {{\mathcal F}}

\def\mcP {{\mathcal P}}
\def\mcS {{\mathcal S}}
\def\mcX {{\mathcal X}}

\def \msL {{\mathscr L}}

\def \a {{\mathfrak a}}

\def \c {{\mathfrak c}}

\def \f {{\mathfrak f}}

\def \p {{\mathfrak p}}

\def \mfD {{\mathfrak D}}

\def \mfI {{\mathfrak I}}

\def \w {{\textnormal{\textbf{w}}}}

\def \bft {{\textnormal{\textbf{t}}}}

\newcommand{\aff}{\textnormal{aff}}

\newcommand{\disc}{\textnormal{disc}}

\newcommand{\GL}{\textnormal{GL}}

\newcommand{\pr}{\textnormal{pr}}

\newcommand{\Aut}{\textnormal{Aut}}

\newcommand{\Ht}{\textnormal{Ht}}

\newcommand{\Tr}{\textnormal{Tr}}

\newcommand{\Val}{\textnormal{Val}}
\newcommand{\ord}{\textnormal{ord}}

\newcommand{\tor}{\textnormal{tor}}
\newcommand{\mult}{\textnormal{mult}}
\newcommand{\RE}{\textnormal{Re}}

\newcommand{\sqf}{\textnormal{sqf}}
\newcommand{\sq}{\textnormal{sq}}

\renewcommand\Im{\operatorname{Im}}

\newcommand{\legendre}[2]{\genfrac{(}{)}{}{}{#1}{#2}}

\newcommand{\defeq}{\stackrel{\textnormal{def}}{=}}

\begin{document}

\title{Average Analytic Ranks of Elliptic Curves\\ over Number Fields}

\author{Tristan Phillips}
\email{tristanphillips72@gmail.com}

\subjclass[2010]{Primary 11G05; Secondary 11G07, 11G35, 11G50, 11D45, 14G05.}

\begin{abstract}
A conditional bound is given for the average analytic rank of elliptic curves over an arbitrary number field. In particular, under the assumptions that all elliptic curves over a number field $K$ are modular and have $L$-functions which satisfy the Generalized Riemann Hypothesis, it is shown that the average analytic rank of isomorphism classes of elliptic curves over $K$ is bounded above by $(9\deg(K)+1)/2$, when ordered by naive height. A key ingredient in the proof is giving asymptotics for the number of elliptic curves over an arbitrary number field with a prescribed local condition; these results are obtained by proving general results for counting points of bounded height on weighted projective stacks with a prescribed local condition, which may be of independent interest. 
\end{abstract}

\maketitle

\section{Introduction}

\subsection{Average analytic ranks of elliptic curves}

Let $E$ be an elliptic curve over a number field $K$. The Mordell--Weil theorem states that the set of $K$-rational points $E(K)$ of $E$ forms a finitely generated abelian group $E(K)\cong E(K)_{\tor}\oplus \Z^r$, where $E(K)_{\tor}$ is the finite \emph{torsion subgroup} and $r\in \Z_{\geq 0}$ is the \emph{rank}. The study of ranks of elliptic curves has become a central topic in number theory. Despite this attention, ranks remain a mystery in many ways.

Birch and Swinnerton-Dyer \cite{BSD65} famously conjectured that the rank of an elliptic curve equals its \emph{analytic rank} (i.e., the order of vanishing of its $L$-function at $s=1$). 

To get an understanding of ranks of elliptic curves in general, one can hope to determine the average rank of elliptic curves. As there are infinitely many elliptic curves, in order to make sense of this average one must order elliptic curves in some way. Throughout this article we order elliptic curves by their \emph{naive height}, as defined at the end of Section \ref{sec:Weighted-Projective-Spaces}. The average analytic rank of elliptic curves over the rational numbers was first bounded by Brumer \cite{Bru92}, who gave an upper bound of $2.3$ under the assumption of the Generalized Riemann Hypothesis for elliptic $L$-functions.\footnote{Brumer also assumed that all elliptic curves over $\Q$ were modular, which was unknown at the time, but later confirmed in a series of papers by Andrew Wiles and his students.}
Under the same conditions, this bound was  improved to $2$ by Heath-Brown \cite{HB04}, and then to $25/14\approx 1.8$ by Young \cite{You06}. In a remarkable series of papers \cite{BS15a, BS15b, BS13a, BS13b} Bhargava and Shankar bounded the average size of $2$-, $3$-, $4$-, and $5$-Selmer groups of elliptic curves, leading to an  unconditional upper bound of $0.885$ for the average rank of elliptic curves over $\Q$.

Surprisingly little is known about average ranks of elliptic curves over number fields beyond $\Q$. In Shankar's doctoral thesis \cite{Sha13}, he extends his work with Bhargava on bounding $2$-Selmer groups to show that the average rank of elliptic curves over number fields is bounded above by $1.5$. However, Shankar's result counts \emph{reduced Weierstrass equations}, which differs from counting isomorphism classes. In particular, counting reduced Weierstrass equations coincides with counting isomorphism classes of elliptic curves only for the finitely many number fields with unit group $\{\pm 1\}$ and class number $1$; namely $\Q$ and the imaginary quadratic extensions $\Q(\sqrt{-d})$ for $d\in \{2, 7, 11, 19, 43, 67, 163\}$. Outside of these cases there can be multiple reduced Weierstrass equations within the same isomorphism class of elliptic curves. 

In this article we prove a conditional bound for the average analytic rank of isomorphism classes of elliptic curves over an arbitrary number field. This appears to be the first known bound on average ranks of isomorphism classes of elliptic curves over arbitrary number fields, as well as the first bound for average analytic ranks of elliptic curves over number fields other than $\Q$.

\begin{theorem}\label{thm:AvgRank}
Let $K$ be a number field of degree $d$. Assume that all elliptic curves over $K$ are modular and that their $L$-functions satisfy the Generalized Riemann Hypothesis. Then the average analytic rank of isomorphism classes of elliptic curves over $K$, when ordered by naive height, is bounded above by $(9d+1)/2$.
\end{theorem}


One of the main difficulties in extending methods for counting elliptic curves over the rational numbers to elliptic curves over more general number fields is that one may no longer have a bijection between reduced short Weierstrass models and isomorphism classes of elliptic curves. We overcome this difficulty by exploiting the geometry of the moduli stack of elliptic curves. In particular, the (compactified) moduli stack of elliptic curves is isomorphic to the weighted projective stack $\mcP(4,6)$. From this perspective questions about counting elliptic curves of bounded height turn into questions about counting points of bounded height on weighted projective stacks. Such questions can then be studied using various tools from Diophantine geometry.

 One of the key ingredients in the proof of Theorem \ref{thm:AvgRank} is estimating the number of elliptic curves with prescribed local conditions over number fields. To state our results on counting elliptic curves, we introduce some notation. Let $K$ be a number field with ring of integers $\O_K$, discriminant $\Delta_K$, class number $h_K$, regulator $R_K$, and which contains $\varpi_K$ roots of unity. Let $\Val(K)$ denote the set of places of $K$, let $\Val_0(K)$ denote the set of finite places, and let $\Val_\infty(K)$ denote the set of infinite places. Let $\zeta_K$ denote the Dedekind zeta function of $K$. Let $H(\cdot)$ denote the Hurwitz-Kronecker class number (see Subsection \ref{subsec:CountingLocalConditions} for the definition). For a prime $\p\subset \O_K$, above a rational prime $p\in \Z$, we will write $\deg(\p)$ for the index of fields $[\O_K/\p:\Z/p]$, which we will refer to as the \emph{degree of $\p$}.

\begin{theorem}\label{thm:LocalConditions}
Let $K$ be a degree $d$ number field, and let $\p$ be a prime ideal of $\O_K$ of norm $q$ such that $2\nmid q$ and $3\nmid q$. Let $\msL$ be one of the local conditions listed in Table \ref{tab:LocalConditions}. Then the number of elliptic curves over $K$ with naive height less than $B$, and which satisfy the local condition $\msL$ at $\p$, is
\[
\kappa\kappa_\msL B^{5/6}+O\left(\epsilon_\msL B^{\frac{5}{6}-\frac{1}{3d}}\right),
\]
where
\[
\kappa=\frac{h_K \left(2^{r_1+r_2}\pi^{r_2}\right)^{2}R_K10^{r_1+r_2-1}\gcd(2,\varpi_K)}{\varpi_{K} |\Delta_K|\zeta_K(10)},
\]
and where $\kappa_\msL$ and $\epsilon_\msL$ are as in Table \ref{tab:LocalConditions}.
\end{theorem}

\begin{table}[h]
\centering
\begin{tabular}{ccc}
\hline\vspace{-11pt}\\
$\msL$ & $\kappa_\msL$ & $\epsilon_\msL$ \\
\hline
\vspace{-7pt}\\
good & $\frac{q-1}{q}\frac{q^{10}}{q^{10}-1}$ & $q$\\
bad & $\frac{q^{9}-1}{q^{10}}\frac{q^{10}}{q^{10}-1}$ & $q$\\
multiplicative & $\frac{q-1}{q^2}\frac{q^{10}}{q^{10}-1}$ & $1$ \\
split multiplicative & $\frac{q-1}{2q^2}\frac{q^{10}}{q^{10}-1}$ & $1$ \\
nonsplit multiplicative & $\frac{q-1}{2q^2}\frac{q^{10}}{q^{10}-1}$  & $1$ \\
additive & $\frac{q^8-1}{q^{10}}\frac{q^{10}}{q^{10}-1}$  & $q$ \\
$\textnormal{I}_m$ & $\frac{q^2-2q+1}{q^{m+2}}\frac{q^{10}}{q^{10}-1}$  & $q$ \\
$\textnormal{II}$ & $\frac{q-1}{q^3}\frac{q^{10}}{q^{10}-1}$  & $1$ \\
$\textnormal{III}$ & $\frac{q-1}{q^4}\frac{q^{10}}{q^{10}-1}$  & $1$ \\
$\textnormal{IV}$ & $\frac{q-1}{q^5}\frac{q^{10}}{q^{10}-1}$  & $1$ \\
$\textnormal{I}_0^\ast$ & $\frac{q-1}{q^6}\frac{q^{10}}{q^{10}-1}$  & $q$ \\
$\textnormal{I}_m^\ast$ & $\frac{q^2-2q+1}{q^{m+7}}\frac{q^{10}}{q^{10}-1}$  & $q$ \\
$\textnormal{II}^\ast$ & $\frac{q-1}{q^{10}}\frac{q^{10}}{q^{10}-1}$  & $1/q^2$ \\
$\textnormal{III}^\ast$ & $\frac{q-1}{q^9}\frac{q^{10}}{q^{10}-1}$  & $1/q^3$ \\
$\textnormal{IV}^\ast$ & $\frac{q-1}{q^8}\frac{q^{10}}{q^{10}-1}$  & $1/q$ \\
\hline
\end{tabular}
\caption{Local conditions}
\label{tab:LocalConditions}
\end{table}
We also give asymptotics for counting elliptic curves with a prescribed trace of Frobenius:

\begin{theorem}\label{thm:aqLocalCondition}
Let $K$ be a degree $d$ number field, and let $\p$ be a degree $n$ prime ideal of $\O_K$ above a rational prime $p>3$. Set $q=N_{K/\Q}(\p)=p^n$, the norm of $\p$. Let $a\in \Z$ be an integer satisfying $|a|\leq 2\sqrt{q}$. Then the number of elliptic curves over $K$ with naive height less than $B$, good reduction at $\p$, and which have trace of Frobenius $a$ at $\p$, is
\[
\kappa\kappa_{n,a} B^{5/6}+O\left(\epsilon_{n,a} B^{\frac{5}{6}-\frac{1}{3d}}\right),
\]
where
\[
\kappa=\frac{h_K \left(2^{r_1+r_2}\pi^{r_2}\right)^{2}R_K10^{r_1+r_2-1}\gcd(2,\varpi_K)}{\varpi_{K} |\Delta_K|\zeta_K(10)},
\]
and where $\kappa_{n,a}$ and $\epsilon_{n,a}$ are as in Table \ref{tab:aqLocalConditions}.
\end{theorem}

\begin{table}[H]
\centering
\begin{tabular}{cccc}
\hline
$n$ & $a$ & $\kappa_{n,a}$ & $\epsilon_{n,a}$ \\
\hline
\text{any} & $|a|< 2\sqrt{q}$ \text{ and } $p\nmid a$ & $\frac{q^{10}}{q^{10}-1}(q-1)H(a^2-4q)/q^2$ & $H(a^2-4q)$\\
\text{odd} & $a=0$ & $\frac{q^{10}}{q^{10}-1}(q-1)H(-4p)/q^2$ & $H(-4p)$\\
\text{ odd } & $a\neq 0$ \text{ and } $p|a$ & $0$ & $0$\\
\text{even} & $a^2=4q$ & $\frac{q^{10}}{q^{10}-1}\frac{q-1}{12q^2}\left(p+6-4\legendre{-3}{p}-3\legendre{-4}{p}\right)$ & $p+6-4\legendre{-3}{p}-3\legendre{-4}{p}$\\
\text{even} & $a^2=q$ & $\frac{q^{10}}{q^{10}-1}(q-1)\left(1-\legendre{-3}{p}\right)/q^2$ & $1-\legendre{-3}{p}$\\
\text{even} & $a^2=0$ & $\frac{q^{10}}{q^{10}-1}(q-1)\left(1-\legendre{-4}{p}\right)/q^2$ & $1-\legendre{-4}{p}$\\
\text{ even} & $a^2\not\in\{4q, q, 0\}$ \text{ and } $p|a$ & $0$ & $0$\\
\hline
\end{tabular}
\caption{Constants for Theorem \ref{thm:aqLocalCondition}}
\label{tab:aqLocalConditions}
\end{table}

Theorem \ref{thm:LocalConditions} and Theorem \ref{thm:aqLocalCondition} generalize a result of Cho and Jeong \cite[Theorem 1.4]{CJ23a} for elliptic curves over $\Q$ to elliptic curves over arbitrary number fields. For more results on counting elliptic curves with prescribed local conditions, see the work of Cremona and Sadek \cite{CS23}.

\subsection{Counting points on weighted projective stacks}

Our results for counting isomorphism classes of elliptic curves are special cases of a more general theorem for counting points of bounded height on weighted projective stacks.

In \cite{Sch79}, Schanuel  proved an asymptotic for the number of rational points of bounded height on projective spaces over number fields. This was generalized to weighted projective stacks by Deng \cite{Den98} (see also the work of Darda \cite{Dar21}, which gives a proof using height zeta functions). In this paper we extend these results to count points of bounded height which satisfy finitely many prescribed local conditions (see Theorem \ref{thm:WProjFinFin}).

\begin{remark}
Let us briefly mention some additional results related to counting points of bounded height on weighted projective stacks.
Recently Bruin and Manterola Ayala \cite{MA21, BMA23} have generalized Deng's result to count points whose images, with respect to a morphism of weighted projective stacks, are of bounded height.
Using the geometric sieve, Bright, Browning, and Loughran \cite{BBL16} have  proven a generalization of Schanuel's theorem which allows one to impose infinitely many local conditions on the points of projective space being counted. In \cite{Phi22} the author combines these results to count points whose images, with respect to a morphism of weighted projective stacks, are of bounded height and satisfy prescribed local conditions. Such a result has applications to counting elliptic curves with certain prescribed level structures and local conditions. Building from this result, Cho, Jeong, and Park \cite{CJP23} have given a conditional bound for average analytic ranks of isomorphism classes of elliptic curves over number fields with a prescribed level structure. 
\end{remark}

\subsection{Asymptotic notation}

Let $f:\R_{>0}\to \R$ and $g:\R_{>0}\to \R$ be real valued functions. Throughout the article we will write $f=O(g)$ (or $f\ll g$) if there exist constants $C,D\in \R_{>0}$ such that for all $x\geq D$ we have $|f(x)|\leq C\cdot g(x)$. We refer to any such $C$ as the implied constant. If the implied constant depends on a parameter $t$, then we will write $f=O_t(g)$ (or $f\ll_t g$).

\subsection{Organization}

In Section \ref{sec:Weighted-Projective-Spaces} we cover basic facts about weighted projective stacks and heights defined on weighted projective stacks. 
In Section \ref{sec:Geometry-of-Numbers} we prove a result for counting points of bounded height satisfying prescribed local conditions on affine spaces over number fields. 
In Section \ref{sec:WPROJ} we prove our main result  for counting points of bounded height on weighted projective stacks (Theorem \ref{thm:WProjFinFin}). 
In Section \ref{sec:EC} we apply Theorem \ref{thm:WProjFinFin} to count elliptic curves with prescribed local conditions. This is where we prove Theorem \ref{thm:LocalConditions} and Theorem \ref{thm:aqLocalCondition}.
Finally, in Section \ref{sec:AvgRanks}, we use the explicit formula for L-functions, together with our results for counting elliptic curves with prescribed local conditions, to bound the average analytic rank of elliptic curves over number fields. This is where we prove Theorem \ref{thm:AvgRank}.

\subsection{Acknowledgements}
The author thanks Brandon Alberts, Peter Bruin, Peter Cho, Jordan Ellenberg, Ben Kane, Irati Manterola Ayala, Steven Miller, Grant Molnar, Filip Najman, Junyeong Park, John Voight,  and David Zureick-Brown for helpful conversations and comments on this work. Special thanks to Bryden Cais and Doug Ulmer for their encouragement throughout this project. The author is especially grateful to Keunyoung Jeong for providing thoughtful feedback and corrections.  Finally, a huge thank you to the anonymous referee for thoroughly reading this article multiple times and pointing out countless corrections and suggestions that greatly improved the accuracy and readability of the article. During the completion of this work the author was supported by the National Science Foundation, via grant DMS-2303011.

\section{Preliminaries on weighted projective stacks}\label{sec:Weighted-Projective-Spaces}

In this section we recall basic facts about weighted projective stacks with an emphasis on heights.

\subsection{Heights on weighted projective stacks}\label{subsec:WPSheights}

Given an $(n+1)$-tuple of positive integers $\w=(w_0,\dots,w_n)$, the \textbf{weighted projective stack} $\mcP(\w)$ is the quotient stack
\[
\mcP(\w)\defeq[(\A^{n+1}-\{0\})/\mathbb{G}_m],
\]
where the multiplicative group scheme, $\mathbb{G}_m$, acts on the punctured affine space, $\A^{n+1}-\{0\}$, as follows:
\begin{align*}
\mathbb{G}_m\times(\A^{n+1}-\{0\}) &\to (\A^{n+1}-\{0\})\\
(\lambda,(x_0,\dots,x_n)) &\mapsto \lambda\ast_\w (x_0,\dots,x_n)\defeq(\lambda^{w_0}x_0,\dots, \lambda^{w_n}x_n).
\end{align*}
In the special case when $\w=(1,\dots,1)$, this recovers the usual projective space $\P^n$. 

The point $[(a_0,\dots,a_n)]\in\mcP(w_0,\dots,w_n)$ has stabilizer $\mu_m$, where $m=\gcd(w_i : a_i\neq 0)$. When $K$ is a field of characteristic zero, we have that $\mu_m$ is finite and reduced over $K$. It follows that, over fields of characteristic zero, $\mcP(w_0,\dots,w_n)$ is a Deligne--Mumford stack.

For any field $F$, let $\mcP(\w)(F)$ denote the set of isomorphism classes of the groupoid of $F$-points of $\mcP(\w)$. More concretely, $\mcP(\w)(F)$ is in canonical bijection with the quotient
\[
(F^{n+1}-\{0\})/F^\times,
\]
where $F^\times$ acts on $F^{n+1}-\{0\}$ by the weighted action
\begin{align*}
F^\times \times (F^{n+1}-\{0\}) &\to (F^{n+1}-\{0\})\\
(\lambda, (a_0,\dots,a_n)) &\mapsto (\lambda^{w_0}a_0,\dots,\lambda^{w_n}a_n).
\end{align*}

Throughout this article, unless otherwise stated, we  assume all weighted projective stacks $\mcP(\w)$ are defined over a number field $K$.

For each finite place $v\in \Val_0(K)$, let $\p_v$ be the corresponding prime ideal, and let $\pi_v$ be a uniformizer for the completion $K_v$ of $K$ at $v$. For $x=(x_0,\dots,x_n)\in K^{n+1}-\{0\}$, set $|x_i|_{\w,v}\defeq|\pi_v|_v^{\lfloor v(x_i)/w_i\rfloor}$, and set
\[
 |x|_{{\w},v}\defeq\begin{cases}
 \max_i\{|x_i|_{\w,v}\} & \text{ if } v\in \Val_0(K)\\
 \max_i\{|x_i|_v^{1/w_i}\} & \text{ if } v\in \Val_\infty(K).
 \end{cases}
 \]
 
 \begin{definition}[Height]\label{def:height}
  The \textbf{(exponential) height} of a point $x=[x_0:\cdots:x_n]\in\mcP(\w)(K)$ is defined as
\[
\Ht_\w(x)\defeq\prod_{v\in \Val(K)} |(x_0,\dots,x_n)|_{\w,v}.
\]
\end{definition}
It is straightforward to check that this height function does not depend on the choice of representative of $x$.

\begin{definition}[Scaling ideal]
Define the \textbf{scaling ideal} $\mfI_{\w}(x)$ of $x=(x_0,\dots,x_n)\in K^{n+1}-\{0\}$ to be the fractional ideal
\[
\mfI_{\w}(x)\defeq\prod_{v\in \Val_0(K)} \p_v^{\min_i\{\lfloor v(x_i)/w_i\rfloor\}}.
\]
\end{definition}

The scaling ideal $\mfI_{\w}(x)$ can be characterized as the intersection of all fractional ideals $\a$ of $\O_K$ such that $x\in \a^{w_0}\times \cdots\times\a^{w_n}\subseteq K^{n+1}$. It has the property that
\[
\mfI_\w(x)^{-1}=\{a\in K : a^{w_i}x_i\in \O_K \text{ for all } i\}.
\]
The height can be written in terms of the scaling ideal as follows:
\[
\Ht_\w([x_0:\dots:x_n])=\frac{1}{N_{K/\Q}(\mfI_\w(x))}\prod_{v\in \Val_\infty(K)} \max_i\{|x_i|_v^{1/w_i}\}.
\]
It is straightforward to check that this agrees with the height from Definition \ref{def:height}.

\begin{remark}
The height $\Ht_\w$ was first defined by Deng \cite{Den98}. In the case of projective spaces, $\P^n=\mcP(1,\dots,1)$, this height corresponds to the usual Weil height. In more geometric terms, this height on weighted projective stacks can be viewed as the `stacky height' associated to the tautological bundle of $\mcP(\w)$ (see \cite[\S 3.3]{ESZB23} for this and much more about heights on stacks). 
\end{remark}

The (compactified) moduli stack of elliptic curves, $\mcX_{\GL_2(\Z)}$, is isomorphic to the weighted projective stack $\mcP(4,6)$. This isomorphism can be given explicitly as
\begin{align*}
\mcX_{\GL_2(\Z)} &\xrightarrow{\sim} \mcP(4,6)\\
y^2=x^3+Ax+B & \mapsto [A:B].
\end{align*}
Under this isomorphism, the tautological bundle on $\mcP(4,6)$ corresponds to the Hodge bundle on $\mcX_{\GL_2(\Z)}$. The usual \emph{naive} height of an elliptic curve $E$ over $K$ with a reduced integral short Weierstrass equation $y^2=x^3+Ax+B$, is defined as
\[
\Ht(E)\defeq\prod_{v\in \Val_\infty(K)} \max\{|A|_v^3,|B|_v^2\}.
\]
One can show that this \emph{naive height} is the same as the \textit{stacky} height on $\mcX_{\GL_2(\Z)}$ with respect to the twelfth power of the Hodge bundle \cite[\S 3.3]{ESZB23}. For \textit{any} short Weierstrass equation $y^2=x^3+ax+b$ of $E$, the naive height is given by
\[
\Ht(E)\defeq \Ht_{(4,6)}([a:b])^{12}.
\]

\section{Weighted geometry of numbers over number fields}\label{sec:Geometry-of-Numbers}

Let $K$ be a number field. In this section we prove asymptotics for the number of $K$-rational points of bounded height with prescribed local conditions on affine spaces with respect to a weighting. For these asymptotics we give a power savings error term which is independent of the local conditions. 

\subsection{O-minimal geometry}\label{subsec:GSo-minimal}

 We briefly cover the basics of o-minimal geometry which will be needed in this article. A general reference for this subsection is \cite[\S 1]{vdD98}.
Let $m_L$ denote Lebesgue measure on $\R^n$.

\begin{definition}[Semi-algebraic set]
A (real) \textbf{semi-algebraic subset} of $\R^n$ is a finite union of sets of the form
\[
\{x\in \R^n: f_1(x)=\cdots=f_k(x)=0 \text{ and } g_1(x)>0,\dots,g_l(x)>0\},
\]
where $f_1,\dots,f_k,g_1,\dots,g_l\in \R[X_1,\dots,X_n]$.
\end{definition}

\begin{definition}[Structure]
A \textbf{structure} is a sequence, $\mcS=(\mcS_n)_{n\in \Z_{>0}}$, where each $\mcS_n$ is a set of subsets of $\R^n$ with the following properties:
\begin{enumerate}[label=(\roman*)]
\item If $A,B\in \mcS_n$, then $A\cup B\in \mcS_n$ and $\R^n - A\in \mcS_n$ (i.e., $\mcS_n$ is a Boolean algebra).
\item If $A\in \mcS_m$ and $B\in \mcS_n$, then $A\times B\in \mcS_{n+m}$.
\item If $\pi:\R^n\to\R^m$ is the projection to $m$ distinct coordinates and $A\in \mcS_n$, then $\pi(A)\in \mcS_m$.
\item All real semi-algebraic subsets of $\R^n$ are in $\mcS_n$.
\end{enumerate}
A subset is \textbf{definable} in $\mcS$ if it is contained in some $\mcS_n$. Let $D\subseteq S_n$. A function $f\colon D\to \R^m$ is said to be \textbf{definable} in $\mcS$ if its graph, $\Gamma(f)\defeq\{(x,f(x)):x\in D\}\subseteq \R^{m+n}$, is definable in $\mcS$.
\end{definition}

Note that the intersection of definable sets is definable by property (i).

\begin{definition}[O-minimal structure]
An \textbf{o-minimal structure} is a structure in which the following additional property holds:
\begin{enumerate}[label=(v)]
\item The boundary of each set in $\mcS_1$ is a finite set of points.
\end{enumerate}
\end{definition}

The class of semialgebraic sets is an example of an o-minimal structure (see, e.g., \cite[\S 2]{vdD98}).
The main structure we will use is $\R_{\exp}$, which is defined to be the smallest structure in which the real exponential function, $\exp\colon\R\to\R$, is definable. Observe that the function $\log(x)$ is definable in $\R_{\exp}$ since its graph,
\[
\Gamma(\log(x))=\{(x,\log(x)) : x\in \R_{>0}\}=\{(\exp(y), y) : y\in \R\}\subset \R^2,
\]
is clearly definable in $\R_{\exp}$.

\begin{theorem}[\cite{Wil96}]\label{thm:RexpO-minimal}
The structure $\R_{\exp}$ is o-minimal.
\end{theorem}

From now on we will call a subset of $\R^n$ \textbf{definable} if it is definable in some o-minimal structure. 


\subsection{Weighted geometry-of-numbers}\label{subsec:GSgeometry-of-numbers}

The following proposition is a version of the Principle of Lipschitz, which gives estimates for the number of lattice points in a weighted homogeneous space.

\begin{proposition}\label{prop:o-Lip}
Let $R\subset \R^{n}$ be a bounded set definable in an o-minimal structure. Let  $\Lambda\subset \R^n$ be a rank $n$ lattice with successive minima $\lambda_1,\dots,\lambda_n$ (with respect to the origin-centered unit ball in $\R^n$). Set 
\[
R(B)=B\ast_{\w}R=\{B\ast_{\w}x : x\in R\}.
\] 
Then, for sufficiently large\footnote{Here, just how large \textit{sufficiently large} is, depends on the lattice $\Lambda$.}  $B$,
\[
\#(\Lambda\cap R(B))=\frac{m_L(R)}{\det \Lambda} B^{|\w|} + O\left(\lambda_n\det(\Lambda)^{-1} B^{|\w|-w_{\min}}\right),
\]
where the implied constant depends only on $R$ and $\w$. 
\end{proposition}

\begin{proof}
This will follow from a general version of the Principle of Lipschitz due to Barroero and Widmer \cite[Theorem 1.3]{BW14}. Since $m_L(R(B))=B^{|\w|} m_L(R)$, \cite[Theorem 1.3]{BW14} gives the desired leading term. Let $V_j(R(B))$ denote the sum of the $j$-dimensional volumes of the orthogonal projections of $R(B)$ onto each $j$-dimensional coordinate subspace of $\R^n$. The error term given in \cite[Theorem 1.3]{BW14} is
\begin{align}\label{eq:o-lip error}
O\left(1+\sum_{j=1}^{n-1} \frac{V_j(R(B))}{\lambda_1\cdots \lambda_j}\right),
\end{align}
where the implied constant depends only on $R$.
 In our case, one observes that $V_i(R(B))=O(V_{j}(R(B)))$ for all $i\leq j$. Moreover, $V_{n-1}(R(B))=O(\sum_{i\leq n} B^{|\w|-w_i})=O(B^{|\w|-w_{\min}})$, where the implied constant depends only on $R$ and $\w$. By dimension considerations, we see that if $B^{|\w|-w_{\min}}=O(V_i(R(B)))$, then we must have $i\geq n-1$. From these observations it follows that
 \[
O\left(1+\sum_{j=1}^{n-1} \frac{V_j(R(B))}{\lambda_1\cdots \lambda_j}\right)=O\left(\frac{V_{n-1}(R(B))}{\lambda_1\cdots\lambda_{n-1}}\right).
\] 
By Minkowski's second theorem \cite{Min07} (for a more modern reference see, e.g., \cite[p. ~203]{Cas59}), 
\[
\lambda_1\cdots\lambda_n\geq \frac{2^n}{n!\cdot m_L(\mathscr{B}_n)}\det(\Lambda),
\]
where $\mathscr{B}_n$ is the unit ball in $\R^n$. From this we obtain the desired expression for the error term,
\[
O\left(\frac{V_{n-1}(R(B))}{\lambda_1\cdots\lambda_{n-1}}\right)=O\left(\frac{\lambda_n B^{|\w|-w_{\min}}}{\det(\Lambda)}\right).
\]
\end{proof}

\begin{definition}[Boxes]\label{def:boxes}
A \textbf{($\Z_p^n$)-box} is a subset $\mcB_p\subset \Z_p^n$ for which there exist closed balls
\[
\mcB_{p,j}=\{x\in \Z_p : |x-a_j|_p\leq b_{p,j}\}\subset \Z_p,
\]
 where $a_j\in \Z_p$ and $b_{p,j}\in \{p^{-k}: k\in \Z_{\geq 0}\}$, such that $\mcB_p$ equals the Cartesian product $\prod_{j=1}^n \mcB_{p,j}$. Let $S$ be a finite set of primes. An \textbf{S-box} is a subset $\mcB\subset \prod_{p\in S} \Z_p^n$ for which there exist $\Z_p^n$-boxes $\mcB_p$ such that
\[
\mcB=\prod_{p\in S} \mcB_{p}=\prod_{p\in S}\prod_{j=1}^n \mcB_{p,j}.
\]
\end{definition}


\begin{lemma}[Box Lemma]\label{lem:box}
Let $\Omega_\infty\subset \R^n$ be a bounded subset definable in an o-minimal structure. Let $S$ be a finite set of primes and $\mcB=\prod_{p\in S} \mcB_p$ an S-box. Then, for each $B\in \R_{>0}$, we have
\begin{align*}
&\#\left\{x\in \Z^n \cap \left(B\ast_\w \Omega_\infty\right): x\in \prod_{p\in S} \mathcal{B}_p \right\}\\
&\hspace{1cm}=\left(m_L(\Omega_\infty)\prod_{p\in S} m_p(\mathcal{B}_p)\right)B^{|\w|} + O\left(\max_j\left\{\prod_{p\in S} b_{p,j}^{-1}\right\}\left(\prod_{p \in S} m_p(\mathcal{B}_p)\right)B^{|\w|-w_{\min}} \right),
\end{align*}
where the implied constant depends only on $\Omega_\infty$ and $\w$.
\end{lemma}

\begin{proof}
We have that $\mathcal{B}_p=\prod_{j=1}^n \mathcal{B}_{p,j}$ with $\mathcal{B}_{p,j}=\{x\in \Z_p : |x-a_j|_p\leq b_{p,j}\}$, where $a_j\in \Z_p$ and $b_{p,j}\in\{p^{-k} : k\in \Z_{\geq 0}\}$. Observe that for each $j\in \{1,\dots,n\}$ the set $\Z\cap \mathcal{B}_{p,j}$ is a translate of the sublattice $\langle b_{p,j}^{-1}\rangle\subseteq \Z$. By the Chinese Remainder Theorem, the set $\Z\cap \prod_{p\in S} \mathcal{B}_{p,j}$ is a translate of the sublattice $\left\langle\prod_{p\in S} b_{p,j}^{-1}\right\rangle\subseteq \Z$. Taking a Cartesian product, we find that the set 
\[
\prod_{j=1}^n\left( \Z\cap \prod_{p\in S} \mathcal{B}_{p,j}\right)=\Z^n\cap \prod_{p\in S} \mathcal{B}_p
\]
 is a translate of a sub-lattice of $\Z^n$ whose successive minima are  $\prod_{p\in S} b_{p,j}^{-1}$. In particular, the $n$-th successive minimum of the sublattice is
\[
\max_j\left\{\prod_{p\in S} b_{p,j}^{-1}\right\},
\]
and the determinant of the sublattice is
\[
\prod_{j=1}^n\prod_{p\in S}  b_{p,j}^{-1}=\prod_{p\in S}m_p(\mathcal{B}_p)^{-1}.
\] 
 The lemma then follows from Proposition \ref{prop:o-Lip}.
\end{proof}

Let $\Lambda$ be a free $\Z$-module of finite rank $n$. Set $\Lambda_\infty=\Lambda\otimes_\Z \R$ and $\Lambda_p=\Lambda\otimes_\Z \Z_p$. Equip $\Lambda_\infty$ and $\Lambda_p$ with Haar measures $m_\infty$ and $m_p$ respectively, normalized so that $m_p(\Lambda_p)=1$ for all but finitely many primes $p$. As $\Lambda_p\cong \Z_p^n$, we define a \textbf{($\Lambda_p)$-box} to be a subset of $\Lambda_p$ isomorphic to a $(\Z_p^n)$-box.

\begin{lemma}\label{lem:ZnCRTfinfin}
Let $\Lambda$ be a free $\Z$-module of finite rank $n$, let $\Omega_\infty\subset \Lambda_\infty$ be a bounded definable subset, and, for each prime $p$ in a finite subset $S$, let $\Omega_p=\prod_j \{x\in \Z_p : |x-a_{p,j}|_p\leq \omega_{p,j}\}\subset \Lambda_p$ be a $(\Lambda_p)$-box.
Let $\w=(w_1,\dots,w_n)$ be an $n$-tuple of positive integers.
Then 
\begin{align*}
&\#\{x\in \Lambda\cap \left(B \ast_\w \Omega_\infty\right): x\in \Omega_p \text{ for all primes } p\in S\}\\
&\hspace{1cm}=
\left(\frac{m_\infty(\Omega_\infty)}{m_\infty(\Lambda_\infty/\Lambda)}\prod_{p\in S} \frac{m_p(\Omega_p)}{m_p(\Lambda_p)}\right)B^{|\w|}
 + O\left(\max_j\left\{\prod_{p\in S} \omega_{p,j}^{-1}\right\}\left(\prod_{p\in S} \frac{m_p(\Omega_p)}{m_p(\Lambda_p)}\right)B^{|\w|-w_{\min}}\right),
\end{align*}
where the implied constant depends only on $\Lambda$, $\Omega_\infty$, and $\w$.
\end{lemma}

\begin{proof}
 Fix an isomorphism $\Lambda\cong \Z^n$. The measures $m_\infty$ and $m_p$ on $\Lambda_\infty$ and $\Lambda_p$ induce measures on $\R^n$ and $\Z_p^n$ which differ from the usual Haar measures by $m_\infty(\Lambda_\infty/\Lambda)$ and $m_p(\Lambda_p)$ respectively. It therefore suffices to prove the result in the case $\Lambda=\Z^n$, and $m_\infty=m_L$ and $m_p$ are the usual Haar measures; but this case is precisely the Box Lemma (Lemma \ref{lem:box}).
\end{proof}
 
If $K$ is a number field of degree $d$ over $\Q$ with discriminant $\Delta_K$, then its ring of integers $\O_K$ may naturally be viewed as a rank $d$ lattice in $K_\infty\defeq\O_K\otimes_{\Z} \R=\prod_{v|\infty} K_v$. This lattice has covolume $|\Delta_K|^{1/2}$ with respect to the usual Haar measure $m_\infty$ on $K_\infty$ (which differs from Lebesgue measure on $K_\infty\cong \R^{r_1+2r_2}$ by a factor of $2^{r_2}$) \cite[Chapter I Proposition 5.2]{Neu99}. 
More generally, any non-zero integral ideal $\a\subseteq \O_K$ may be viewed as a lattice in $K_\infty$ with covolume $N_{K/\Q}(\a)|\Delta_K|^{1/2}$. For an $n$-tuple of positive integers $\w=(w_1,\dots,w_n)$, define the lattice
\[
\a^\w\defeq\a^{w_1}\times\cdots \times \a^{w_n}\subset K_\infty^{n},
\]
where we view each $\a^{w_i}$ as a subset of $K_\infty$.
This lattice has covolume $N_{K/\Q}(\a)^{|\w|}|\Delta_K|^{n/2}$. For example, in the case that $\a=\O_K$ and $w_i=1$ for all $i$, one has the lattice $\a^\w=\O_K^{n}$ of covolume $|\Delta_K|^{n/2}$.

 For any rational prime $p$, we have $\O_K\otimes_{\Z} \Z_p=\prod_{\p|p} \O_{K,\p}$. Equip each $\O_{K,\p}$ with the Haar measure $m_\p$, normalized so that $m_\p(\O_{K,\p})=1$. This measure induces a measure on $\O_{K,\p}^n$, which we will also denote by $m_\p$. Similarly, the measure $m_\infty$ on $K_\infty$ induces a measure on $K_\infty^n$, which we will denote by $m_\infty$. 
 
 For $v\in \Val_0(K)$ a finite place, let $q_v$ denote the size of the residue field $\O_{K}/\p_v$. Let $\a_\p$ denote the image of $\a$ in $\O_{K,\p}$.
 
\begin{definition}[Local condition]\label{def:local-condition}
 An \textbf{(affine) local condition} at a finite place $v\in\Val_0(K)$ will refer to a subset $\Omega_v\subseteq \O_{K,\p_v}^n$. We call a local condition $\Omega_v$ \textbf{irreducible} if it is an ($\O_{K,\p_v}^n$)-box, that is, if there exist $a_{j}\in \O_{K,\p_v}$ and $\omega_{j}\in \{q_v^{-k} : k\in \Z_{\geq 0}\}$ such that $\Omega_v=\prod_j \{ x\in \O_{K,\p_v} : |x-a_j|_v\leq \omega_j\}\subset\O_{K,\p_v}^n$.
 \end{definition}
 
 \begin{remark}
Though our results will mostly concern irreducible local conditions, by taking unions and complements one can easily extend these results to many other local conditions. 
 \end{remark}

We now apply Lemma \ref{lem:ZnCRTfinfin} to the rank $nd$ lattice $\a^\w$, and with the $(nd)$-tuple of weights $\tilde{\w}=(w_1,\dots,w_1,w_2,\dots,w_2,\dots,w_n,\dots,w_n)$, obtained by repeating each of the weights in the $n$-tuple $\w$ exactly $d$ times. Observing that $|\tilde{\w}|=d|\w|$ and $\tilde{w}_{\min}=w_{\min}$, we obtain the following proposition:

\begin{proposition}\label{prop:KCRTfinfin}
 Let $\Omega_\infty\subset K_\infty^{n}$ be a bounded definable subset. Let $S$ be a finite set of prime ideals of $\O_K$. For each $\p\in S$ let $\Omega_\p=\prod_j \{x\in \O_{K,\p} : |x-a_{\p,j}|_\p\leq \omega_{\p,j}\}\subset \O_{K,\p}^n$ be an irreducible local condition. 
Then 
\begin{align*}
\#\{x\in \a^\w\cap \left(B \ast_\w \Omega_\infty\right) : x\in \Omega_\p \text{ for all primes } \p\in S\}=\kappa B^{d|\w|}+ 
 O\left(\epsilon B^{d|\w|-w_{\min}} \right),
\end{align*}
where
\begin{align*}
\kappa&=\frac{m_\infty(\Omega_\infty)}{N_{K/\Q}(\a)^{|\w|}|\Delta_K|^{n/2}} \left(\prod_{\p\in S} \frac{m_\p(\Omega_\p\cap \a_\p^\w)}{m_\p(\a_\p^\w)}\right),\\
\epsilon&=\max_j\left\{\prod_{\p\in S} \omega_{\p,j}^{-1}\right\}\left(\prod_{\p\in S} \frac{m_\p(\Omega_\p\cap \a_\p^\w)}{m_\p(\a_\p^\w)}\right),
\end{align*}
and
where the implied constant depends only on $K$, $\Omega_\infty$, and $\w$.
\end{proposition}

\section{Counting points on weighted projective stacks}\label{sec:WPROJ}

In this section we prove our results for counting points of bounded height on weighted projective stacks.


\begin{definition}[Local condition]\label{def:projective-local-condition}
 A \textbf{(projective) local condition} at a place $v\in\Val(K)$ will refer to a subset $\Omega_v\subseteq \mathcal{P}(\w)(K_v)$.
 \end{definition}

For a projective local condition $\Omega_v\subseteq \mcP(\w)(K_v)$ we denote the affine cone of $\Omega_v$ by $\Omega_v^{\aff}\subseteq (\A^{n+1}-\{0\})(K_v)$, i.e., $\Omega_v^{\aff}$ is the preimage of $\Omega$ with respect to the map $(\A^{n+1}-\{0\})\to\mcP(\w)$. For $v\in \Val_0(K)$, the set $\Omega_v^{\aff}\cap \O_{K,\p_v}^{n+1}$ is an affine local condition, but not an \emph{irreducible} affine local condition (see Definition \ref{def:local-condition}).

\begin{example}
Let $v\in \Val_0(K)$ be a finite place. Consider the projective local condition 
\[
\Omega_v=\{[a:b]\in \mcP(1,1)(K_v):v(a)=0,\ v(b)=1\}.
\]
Then 
\[
\Omega_v^{\aff}\cap \O_{K,v}^2=\{(\lambda a, \lambda b)\in \O_{K,v}^2 : v(a)=0,\ v(b)=1,\ \lambda\in  K_v^\times\}.
\]
Though this is not an irreducible affine local condition, it contains the irreducible affine local condition 
\[
\Omega_{v,0}^{\aff}\defeq\{(a,b)\in \O_{K,v}^2 : v(a)=0,\ v(b)=1\}.
\]
 Letting $\pi_v$ be a uniformizer of $\O_{K,v}$, we have 
\[
\Omega_v^{\aff}\cap \O_{K,v}^2=\bigcup_{t\geq 0}\left(\pi_v^t\ast_{(1,1)} \Omega_{v,0}^{\aff}\right),
\]
where $\pi_v^t \ast_{(1,1)} \Omega_{v,0}^{\aff}$ is an irreducible affine local condition for each $t$.
\end{example}
This example motivates the following definition:

\begin{definition}[Irreducible projective local condition]\label{def:irreducible-local-condition}
A projective local condition $\Omega_v\subseteq \mcP(\w)(K_v)$ is said to be \textbf{irreducible} if there exists an irreducible affine local condition $\Omega_{v,0}^{\aff}$ such that
\[
\Omega_v^{\aff}\cap \O_{K,v}^{n+1}=\bigcup_{t\geq 0}\left(\pi_v^t\ast_{\w} \Omega_{v,0}^{\aff}\right).
\] 
\end{definition}

In analogy to the archimedean setting, for any ideal $\a_v\subseteq \O_{K,v}$, set  
\[
\a_v^\w\defeq \a_v^{w_0}\times\cdots \times \a_v^{w_n}\subset K_v^{n+1}.
\]

\begin{proposition}\label{prop:irreducible-projective-local-condition}
Let $\Omega_v\subsetneq \mcP(\w)(K_v)$ be a non-trivial irreducible projective local condition with
\[
\Omega^{\aff}_{v,0}=\prod_{j=0}^n\{x\in\O_{K,v} : |x-a_{j}|_v\leq \omega_j\}.
\]
For $t\in \Z_{\geq 0}$, set $\Omega^{\aff}_{v,t}=\pi_v^t\ast_\w \Omega^{\aff}_{v,0}$. Then we have the following:
\begin{enumerate}
\item[(a)]\label{prop:irreducible-projective-local-condition(a)} There exists a $j$ for which $|a_j|_v > \omega_j$.
\item[(b)]\label{prop:irreducible-projective-local-condition(b)} For $t\neq s$ we have $\Omega^{\aff}_{v,t}\cap\Omega^{\aff}_{v,s}=\emptyset$.
\item[(c)]\label{prop:irreducible-projective-local-condition(c)} For $s,t\in \Z_{\geq 0}$ with $s<t$ we have $\Omega_{v,s}^{\aff}\cap(\pi_v^{t})^{\w}=\emptyset$.
\end{enumerate}
\end{proposition}

\begin{proof} (a)
We first prove the contrapositive of part (a). 
Suppose that $|a_j|_v\leq \omega_j$ for all $j$. If each $\omega_j=1$ then $\Omega^{\aff}_{v,0}=\O_{K,v}^{n+1}=\Omega_{v}^{\aff}\cap \O_{K,v}^{n+1}$, but then $\Omega_v=\mcP(\w)(K_v)$ would be the trivial projective local condition. Therefore there must exist an index $j$ for which $\omega_j\neq 1$. By possibly reordering, we may assume $\omega_1\neq 1$. As $|a_1|_v\leq \omega_1<1$, we have that $|1-a_1|_v=1>\omega_1$ (by the ultrametric triangle inequality). Let $b\in \O_{K,v}-\{0\}$ be such that $|b-a_1|_v\leq \omega_1$. Then $(b,0,\dots,0)$ and $(1,0,\dots,0)$ are points in the affine cone $K_v^{n+1}-\{0\}$ of $\mcP(\w)(K_v)$, each above the point $[1:0:\cdots:0]\in \mcP(\w)(K_v)$, but only one is contained in $\Omega_v^{\aff}$. We conclude that $\Omega_v^{\aff}$ cannot be the entire subset of the affine cone above a projective local condition.

(b) We now show that if $|a_j|_v > \omega_j$ for some $j$, then for $t\neq s$ we have $\Omega^{\aff}_{v,t}\cap\Omega^{\aff}_{v,s}=\emptyset$.
Observe that 
\[
\Omega^{\aff}_{v,t}=\prod_{j=0}^n \{x\in \O_{K,v} : |x-\pi_v^{tw_j} a_j|_v\leq q_v^{-tw_j}\omega_j\}
\]
is a product of $v$-adic balls centered at $\pi_v^{tw_j}a_j$ with radius $q_v^{-tw_j}\omega_j$. Without loss of generality we may assume $s< t$. Using $|a_j|_v>\omega_j$, we have that 
\[
|\pi_v^{tw_j}a_j-\pi_v^{sw_j}a_j|_v=q_v^{-sw_j}\cdot |a_j|_v>q_v^{-sw_j} \omega_j>q_v^{-tw_j}\omega_j.
\]
In particular, the distance between the centers of the balls
\[
\{x\in \O_{K,v} : |x-\pi_v^{tw_j} a_j|_v\leq q_v^{-tw_j}\omega_j\} \text{ and } \{x\in \O_{K,v} : |x-\pi_v^{sw_j} a_j|_v\leq q_v^{-sw_j}\omega_j\}
\]
is larger than either of the radii, and thus the balls must be disjoint (since we are in an ultrametric space).

(c) We will show that if $\Omega_{v,s}^{\aff}\cap(\pi_v^{t})^{\w}$ is non-empty, then $\Omega_v$ is not a projective local condition. Let $y=(y_0,\dots,y_n)\in \Omega_{v,s}^{\aff}\cap(\pi_v^{t})^{\w}$ and let $y'=(y'_0,\dots,y'_n)\in \O_{K,v}^{n+1}$ be such that $y=\pi_v^t\ast_\w y'$. It will suffice to show that $y'$ is not in $\Omega^{\aff}_v$.

By part (a) we may choose $j$ to be such that $|a_j|_v>\omega_j$. From this, together with the fact that $y_j\in \pi_v^{t w_j}$ satisfies $|y_j-\pi_v^{sw_j} a_j|_v\leq q_v^{-sw_j}\omega_j$, implies that we must have $|y_j|_v=q_v^{-sw_j}|a_j|_v$. Therefore $|y'_j|_v=q_v^{(t-s)w_j}|a_j|_v$. Since $s<t$, this implies that $|y'_j|_v\neq |\pi_v^{rw_j}a_j|_v$ for any $r\in \Z_{\geq 0}$, and thus $y'$ is not contained in $\Omega^\aff_v\cap \O_{K,v}^{n+1}=\bigcup_t \Omega^{\aff}_{v,t}$. 
\end{proof}

For any set of projective local conditions $(\Omega_v)_{v}$, define the sets
\begin{align*}
\Omega&\defeq\{x\in \mcP(\w)(K) : x\in \Omega_v \text{ for all } v\in \Val(K)\}
,\\
\Omega_\infty&\defeq\{x\in \mcP(\w)(K) : x\in \Omega_v \text{ for all } v\in \Val_\infty(K)\},\\
\Omega_0&\defeq\{x\in \mcP(\w)(K) : x\in \Omega_v \text{ for all } v\in \Val_0(K)\}.
\end{align*}
Set $\varpi_{K,\w}\defeq \varpi_K/\gcd(\varpi_K,w_0,\dots,w_n)$, where we recall that $\varpi_K$ is the number of roots of unity in $K$.

\begin{theorem}\label{thm:WProjFinFin}
Let $K$ be a degree $d$ number field over $\Q$. Let $S\subset \Val_0(K)$ be a finite set of finite places. For each $v\in \Val_\infty(K)\cup S$ let $\Omega_v\subset \mcP(\w)(K_v)$ be a projective local condition. Suppose that $\Omega^\aff_\infty$ is a bounded definable subset of $K_\infty^{n+1}$ and that $\Omega_v$ is irreducible for all $v\in S$ with 
\[
\Omega_{v,0}^{\aff}=\prod_{j=0}^n\{x\in \O_{K,v}: |x-a_{v,j}|_v\leq \omega_{v,j}\},
\]
where $a_{v,j}\in \O_{K,v}$ and $\omega_{v,j}\in \{q_v^{-k}:k\in \Z_{\geq 0}\}$ (see Definitions \ref{def:irreducible-local-condition} and \ref{def:local-condition}). Then
\begin{align*}
\#\{x\in \mcP(\w)(K): \Ht_{\w}(x)\leq B, x\in \Omega\}
 = \kappa B^{|\w|}
 + \begin{cases}
O\left(\epsilon_\Omega B\log(B)\right) &  {\substack{\text{ if } \w=(1,1)\\ \text{and}\ K=\Q,}}\\
O\left(\epsilon_\Omega B^{\frac{d|\w|-w_{\min}}{d}}\right) & \text{ else, }
\end{cases}
\end{align*}
where the leading coefficient is
\[
\kappa=\kappa_\Omega\frac{h_K \left(2^{r_1+r_2}\pi^{r_2}\right)^{n+1}R_K|\w|^{r_1+r_2-1}}{\varpi_{K,\w}|\Delta_K|^{(n+1)/2}\zeta_K(|\w|)},
\]
where
\[
\kappa_\Omega=\prod_{v|\infty} \frac{m_v(\{x\in \Omega_v^\aff : \Ht_{\w,v}(x)\leq 1\})}{m_v(\{x\in K_v^{n+1}: \Ht_{\w,v}(x)\leq 1\})}
\prod_{v\in S} m_v(\Omega_v^\aff\cap \O_{K,v}^{n+1}),
\]
where the factor in the error term is 
\[
\epsilon_{\Omega}=\sum_{t=0}^\infty \max\limits_j\left\{\prod\limits_{v\in S} \frac{N_{K/\Q}(\p_v)^{tw_j}}{\omega_{v,j}}\right\}\frac{1}{N_{K/\Q}(\p_v)^{t|\w|}}\prod\limits_{v\in S} m_v(\Omega_{v,0}^\aff),
\]
and where the implied constant is independent of the local conditions $\Omega_{v}$ with $v\in S$.
\end{theorem}

\begin{proof}
Let $\c_1,\dots,\c_h$ be a set of integral ideal representatives of the ideal class group of $K$. Then we get the following partition of $\mcP(\w)(K)$ into points whose scaling ideals are in the same ideal class:
\[
\mcP(\w)(K)=\bigsqcup_{i=1}^h \{x\in \mcP(\w)(K): [\mfI_\w(x)]=[\c_i]\}.
\] 
This is well defined since the ideal class of a scaling ideal $[\mfI_\w(x)]$ does not depend on the representative of $x$ \cite[Proposition 3.3]{Den98}.

For each $\c\in \{\c_1,\dots,\c_h\}$, consider the counting function
\[
M(\Omega,\c,B)\defeq\#\{x\in \mcP(\w)(K): \Ht_\w(x)\leq B,\ [\mfI_\w(x)]=[\c],\ x\in \Omega\}.
\]
Note that
\[
\#\{x\in \mcP(\w)(K): \Ht_{\w}(x)\leq B, x\in \Omega\}=\sum_{i=1}^{h} M(\Omega,\c_i,B).
\]

Following Schanuel \cite{Sch79}, we shall find an asymptotic for $M(\Omega,\c,B)$ by counting integral points in affine space and then moding out by an action of the unit group, followed by an action of principal ideals.

Consider the (weighted) action of the unit group $\O_K^\times$ on $K^{n+1}-\{0\}$ given by 
\[
u\ast_{\w}(x_0,\dots,x_n)=(u^{w_0}x_0,\dots,u^{w_n}x_n),
\]
 and let $(K^{n+1}-\{0\})/\O_K^\times$ denote the corresponding set of orbits. Let $\Omega^\aff$ be the affine cone of $\Omega$. We may describe $M(\Omega,\c,B)$ in terms of $\O_K^\times$-orbits of an affine cone. In particular, there is a bijection between
\[
\{[x_0:\cdots:x_n]\in \mcP(\w)(K): \Ht_\w(x)\leq B,\ [\mfI_\w(x)]=[\c],\ x\in \Omega\}
\]
and
\[
 \left\{[(x_0,\dots,x_n)]\in (K^{n+1}-\{0\})/\O_K^\times : \frac{\Ht_{\w,\infty}(x)}{N_{K/\Q}(\c)}\leq B,\ \mfI_\w(x)=\c,\ x\in \Omega^\aff\right\}
 \]
 given by
 \[
 [x_0:\dots:x_n] \mapsto [(x_0,\dots,x_n)],
 \]
and therefore
\[
M(\Omega,\c,B)=\#\left\{x\in (K^{n+1}-\{0\})/\O_K^\times : \frac{\Ht_{\w,\infty}(x)}{N_{K/\Q}(\c)}\leq B,\ \mfI_\w(x)=\c,\ x\in \Omega^\aff\right\}.
\]
Our general strategy will be to first find an asymptotic for the counting function
\[
M'(\Omega,\c,B)\defeq\#\left\{x\in (K^{n+1}-\{0\})/\O_K^\times : \frac{\Ht_{\w,\infty}(x)}{N_{K/\Q}(\c)}\leq B,\ \mfI_\w(x)\subseteq\c,\ x\in \Omega^\aff\right\},
\]
and then use M\"obius inversion to obtain an asymptotic formula for $M(\Omega,\c,B)$. 

 We are now going to construct a fundamental domain for the ($\w$-weighted) action of the unit group $\O_K^{\times}$ on $K_\infty^{n+1}-\{0\}$. This will be done using Dirichlet's Unit Theorem. Let $\varpi(K)$ denote the set of roots of unity in $K$.

\begin{theorem}[Dirichlet's Unit Theorem]
The image $\Lambda$ of the map
\begin{align*}
\lambda\colon \O_K^\times &\to \R^{r_1+r_2}\\
u &\mapsto (\log |u|_v)_{v\in \Val_\infty(K)}
\end{align*}
is a rank $r\defeq r_1+r_2-1$ lattice in the hyperplane $H$ defined by $\sum_{v\in \Val_\infty(K)} x_v=0$ and $\ker(\lambda)=\varpi(K)$.
\end{theorem}

For each $v\in \Val_\infty(K)$, define a map
\begin{align*}
\eta_v\colon K_v^{n+1}-\{0\} &\to \R \\
(x_0,\dots,x_n) &\mapsto \log\max_i |x_i|_v^{1/w_i}.
\end{align*}
Combine these maps to obtain a single map
\begin{align*}
\eta\colon \prod_{v\in \Val_\infty(K)} (K_v^{n+1}-\{0\}) &\to \R^{r_1+r_2}\\
x & \mapsto (\eta_v(x_v)).
\end{align*}
Let $H$ be the hyperplane in Dirichlet's unit theorem, and let
\begin{align*}
\pr\colon \R^{r_1+r_2} \to H
\end{align*}
be the projection along the vector $(d_v)_{v\in \Val_\infty(K)}$, where $d_v=1$ if $v$ is real and $d_v=2$ if $v$ is complex. More explicitly,
\[
(\pr(x))_v = x_v-\left(\frac{1}{d}\sum_{v'\in\Val_\infty(K)} x_{v'}\right) d_v.
\]
The reason for choosing this projection is to ensure that a certain expanding region (which we are about to construct) will be weighted homogeneous (see Lemma \ref{lem:F(B)WHom}).

Let $\{u_1,\dots,u_r\}$ be a basis for the image of $\O_K^\times$ in the hyperplane $H$, and let $\{\check{u}_1,\dots,\check{u_r}\}$ be the dual basis. Then the set
\[
\tilde{\mcF}\defeq\{y\in H: 0\leq \check{u}_j(y)<1 \text{ for all }j\in\{1,\dots,r\}\}
\]
is a fundamental domain for $H$ modulo $\Lambda$, and $\mcF\defeq(\pr\circ\eta)^{-1}\tilde{\mcF}$ is a fundamental domain for the ($\w$-weighted) action of $\O_K^\times$ on $\prod_{v\in \Val_\infty(K)}(K_v^{n+1}-\{0\})$. 

Define the sets
\[
\mcD(B)\defeq\left\{x\in \prod_{v\in \Val_\infty(K)} (K_v^{n+1}-\{0\}) : \Ht_{\w,\infty}(x)
\leq B\right\}
\]
and $\mcF(B)\defeq\mcF\cap \mcD(B)$. The sets $\mcD(B)$ are $\O_K^\times$-stable, in the sense that if $u\in \O_K^{\times}$ and $x\in \mcD(B)$ then $u\ast_{\w} x\in \mcD(B)$; this can be seen by the following computation: 
\begin{align*}
\Ht_{\w,\infty}(u\ast_\w x)&=\prod_{v|\infty} \max_i |u^{w_i}x_{v,i}|_v^{1/w_i}\\
&=\prod_{v|\infty} |u|_v \prod_{v|\infty}\max_i |x_{v,i}|_v^{1/w_i}\\
&=\prod_{v|\infty}\max_i |x_{v,i}|_v^{1/w_i}\\
&=\Ht_{\w,\infty}(x).
\end{align*}
 Similarly, for any $t\in \R$, we have that
\[
\Ht_{\w,\infty}(t\ast_\w x)=\prod_{v|\infty} \max_i |t^{w_i}x_{v,i}|_v^{1/w_i}=\prod_{v|\infty} |t|_v \prod_{v|\infty}\max_i |x_{v,i}|_v^{1/w_i}=|t|^{d}\Ht_{\w,\infty}(x).
\]
This shows that $\mcD(B)=B^{1/d}\mcD(1)$ for all $B>0$. 

On the other hand, $\mcF$ is stable under the weighted action of $t\in \R^\times$, in the sense that $t\ast_{\w} \mcF=\mcF$. To see this, note that for any $x\in \prod_{v|\infty} (K_v^{n+1}-\{0\})$,
\[
\eta(t\ast_{\w} x)=(d_v)_{v|\infty}\log(|t|)+\eta(x).
\]
Since the projection $\pr$ is linear and annihilates the vector $(d_v)_{v|\infty}$, we have that
\[
\pr\circ\eta(t\ast_{\w}x)=\pr\circ\eta(x),
\]
as desired. From our observations we obtain the following lemma:

\begin{lemma}\label{lem:F(B)WHom}
The regions $\mcF(B)$ are weighted homogeneous, in the sense that $\mcF(B)=B^{1/d}\ast_{\w} \mcF(1)$ for all $B>0$.
\end{lemma}

We are now going to count lattice points in $\mcF(B)$. In \cite{Sch79} this is done by using the classical Principle of Lipschitz \cite{Dav51} (see also \cite[VI \S 2 Theorem 2]{Lan94}). One of the cruxes of Schanuel's argument is verifying that a fundamental domain (analogous to our $\mcF(1)$) has  Lipschitz parameterizable boundary, so that he can apply the Principle of Lipschitz. Though one can modify this part of Schanuel's argument to work in our case, we will instead take a slightly different route, using an o-minimal version of the Principle of Lipschitz (Proposition \ref{prop:o-Lip}). This allows one to give a more streamlined proof of this part of Schanuel's argument, which may be useful in future generalizations.

\begin{lemma}\label{lem:F(1)Bounded}
The set $\mcF(1)$ is bounded.
\end{lemma}

\begin{proof}
Let $\tilde{H}\subset \prod_{v|\infty} \R_v$ be the subset defined by $\sum_{v|\infty} x_v\leq 0$. Note that $\mcF(1)=\eta^{-1}(\tilde{H} \cap  \pr^{-1}(\tilde{\mcF}))$. It follows from the definition of $\eta$ that, in order to show $\mcF(1)$ is bounded, it suffices to show that the set 
\[
S\defeq\tilde{H} \cap  \pr^{-1}(\tilde{\mcF})\subset \R^{r_1+r_2}
\]
 is bounded above (i.e., there exist constants $c_1,\dots,c_{r_1+r_2}$ such that for each 
 \[
 x=(x_1,\dots,x_{r_1+r_2})\in S
 \]
  we have $x_i\leq c_i$ for all $i$). For this, note that any $x\in S$ can be written as 
\[
x=\pr(x) + (d_v)_v \frac{1}{d} \sum_{v'|\infty} x_{v'},
\]
and thus the components of $S$ are bounded above, noting that the first term, $\pr(x)$, has components bounded above, and the second term has negative components. 
\end{proof}

\begin{lemma}\label{lem:F(1)Definable}
The set $\mcF(1)$ is definable in $\R_{\exp{}}$.
\end{lemma}

\begin{proof}
We make the following straightforward observations:
\begin{itemize}
\item The product $\prod_{v|\infty} (K_v^{n+1}-\{0\})$, viewed as a subset of $\R^{(r_1+2r_2)(n+1)}$, is semi-algebraic since it is clearly the complement of a semi-algebraic set. 
\item The set $\mcD(1)$ is semi-algebraic.
\item The set\\
\resizebox{\linewidth}{!}{
  \begin{minipage}{\linewidth}
  \begin{align*}
&\mcF=\{x\in \prod_{v|\infty} (K_v^{n+1}-\{0\}) : \pr\circ\eta(x)\in \tilde{\mcF}\}\\
&=\left\{x\in \prod_{v|\infty} (K_v^{n+1}-\{0\}) : 0\leq \check{u}_j \left(\log(\max_i |x_i|_v^{1/w_i})-d_v\left(\frac{1}{d} \sum_{v|\infty}\log(\max_i |x_i|_{v'}^{1/w_i})\right)\right)_v <1 \forall j\right\}
\end{align*} 
\end{minipage}}\\
is definable in $\R_{\exp}$, since it can be described in terms of polynomials and $\log$, and $\log$ is definable in $\R_{\exp}$ (as observed in Subsection \ref{subsec:GSo-minimal}). 
\end{itemize}

It follows that the intersection $\mcF(1)=\mcF\cap\mcD(1)$ is definable in $\R_{\exp}$.
\end{proof}

We now analyze $M'(\Omega,\c,B)$. Note that if $\mfI_\w(x)\subseteq \c$, then we must have that $x$ is contained in the lattice
 \[
\c^\w=\c^{w_0}\times\cdots\times \c^{w_n}.
\]
Therefore, we can write $M'(\Omega,\c,B)$ as
\[
M'(\Omega,\c,B)=\#\left\{x\in (\c^\w-\{0\})/\O_K^\times : \frac{\Ht_{\w,\infty}(x)}{N_{K/\Q}(\c)}\leq B,\  x\in \Omega^\aff\right\}.
\]

We now stratify $\Omega$. For each $t\in \Z_{\geq 0}$, set 
\[
\Omega_{v,t}^{\aff}=\pi_v^t\ast_\w \Omega_{v,0}^{\aff}=\prod_{j=0}^n \{x\in \O_{K,v} : |x- \pi_v^{w_j t} a_{v,j}|_v\leq q_v^{-w_j t}\omega_{v,j}\}.
\]
For $\bft=(t_v)_{v\in S} \in \Z_{\geq 0}^{\# S}$, set
\[
\Omega_{\bft}^{\aff}\defeq \{x\in \Omega^{\aff} : x\in \Omega_{v,t_v}^{\aff} \text{ for all } v\in S\}.
\]
Since each $\Omega_v$ is irreducible, we have
\[
\Omega^{\aff} \cap \O_{K,v}^{n+1} = \bigsqcup_{\bft\in \Z^{\# S}_{\geq 0}} \Omega_\bft^{\aff}.
\]
Set  
\[
 \epsilon_{\bft}=\max_j\left\{\prod_{\p\in S} q_v^{t_vw_j}\omega_{v,j}^{-1}\right\}\left(\prod_{\p\in S} \frac{m_\p(\Omega_{v,t_v}^{\aff}\cap \c_\p^\w)}{m_\p(\c_\p^\w)}\right).
 \]

Note that each $\varpi(K)$-orbit (with respect to the $\w$-weighted action) of an element of $(K-\{0\})^{n+1}$ contains $\varpi_{K,\w}$ elements. Thus 
\begin{align}\label{eq:orbits}
M'(\Omega, \c, B)
=\sum_{\bft\in \Z^{\# S}_{\geq 0}} 
\left(\frac{\# \left\{ \Omega^{\aff}_{\bft} \cap \mcF\left(N_{K/\Q}(\c) B\right)\cap \c^\w\right\}}{\varpi_{K,\w}}
 + O\left(\epsilon_\bft B^{|\w|-w_{\min}/d}\right)\right),
\end{align}
where each error term accounts for the points of $M'(\Omega,\c, B)$ contained in the intersection of $\Omega_\bft^{\aff}$ and the subvariety of the affine cone of $\mcP(\w)(K)$ consisting of points with at least one coordinate equal to zero; by Proposition \ref{prop:KCRTfinfin} this error is at most $O\left(\epsilon_\bft \left(B^{1/d}\right)^{d|\w|-w_{\min}}\right)$.

By Lemma \ref{lem:F(1)Bounded} and Lemma \ref{lem:F(1)Definable}, we may apply Proposition \ref{prop:KCRTfinfin}, with $\a=\c$ and the local conditions $\Omega^{\aff}_v\cap\mcF(1)$ for $v|\infty$ and $\Omega^{\aff}_{v,t}$ for $v\in S$. Doing so, we have 
  \[
 \#\{\Omega_{\bft}^{\aff}\cap \mcF(N_{K/\Q}(\c)B)\cap \c^{\w}\}=\frac{m_\infty(\Omega^\aff_\infty\cap \mcF(1))}{|\Delta_K|^{(n+1)/2}} \left(\prod_{\p\in S} \frac{m_\p(\Omega^{\aff}_{v,t_v}\cap\c_\p^\w)}{m_\p(\c_\p^\w)}\right)B^{|\w|} + O\left(\epsilon_{\bft} B^{|\w|-\frac{w_{\min}}{d}}\right).
 \] 
 
Let $v(\c)$ denote the $v$-adic valuation of the ideal $\c\subseteq \O_K$ (i.e., the valuation of any (and all) embeddings of $\c$ into $\O_{K,v}$). By Proposition \ref{prop:irreducible-projective-local-condition}(c), if $t< v(\c)$, then $m_\p(\Omega^{\aff}_{v,t}\cap\c_\p^\w)=0$. On the other hand, when $t\geq v(\c)$ we compute
\begin{align*}
\frac{m_\p(\Omega^{\aff}_{v,t}\cap\c_\p^\w)}{m_\p(\c_\p^{\w})}
&=\frac{1}{m_\p(\c_\p^\w)} m_\p\left(\prod_{j=0}^n \{x\in \c_\p^{w_j} : |x-\pi_v^{tw_j}a_j|_v\leq q_v^{-tw_j}\omega_j\}\right)\\
&=\frac{1}{m_\p(\c_\p^\w)} m_\p\left(\prod_{j=0}^n \{x\in \O_{K,\p} : |\pi_v^{v(\c)w_j}x-\pi_v^{tw_j}a_j|_v\leq q_v^{-tw_j}\omega_j\}\right)\\
&=m_\p\left(\Omega^{\aff}_{v,t-v(\c)}\right).
\end{align*}
By Proposition \ref{prop:irreducible-projective-local-condition} we know that the $\Omega_{v,t}^{\aff}$ are disjoint for distinct values of $t$. Therefore, summing over all $t\in \Z_{\geq 0}$, it follows that
\[
\sum_{t\geq 0} \frac{m_\p(\Omega^{\aff}_{v,t}\cap\c_\p^\w)}{m_\p(\c_\p^{\w})}=\sum_{s\geq 0} m_\p(\Omega_{v,s}^{\aff})=m_\p(\Omega_v^{\aff}\cap \O_{K,v}^{n+1}),
\]
and we obtain the asymptotic
\begin{align}\label{eq:I<=c asymptotic}
M'(\Omega,\c,B)=\frac{m_\infty(\Omega^\aff_\infty\cap \mcF(1))}{\varpi_{K,\w}|\Delta_K|^{(n+1)/2}} \left(\prod_{\p\in S} m_\p(\Omega^{\aff}_{v}\cap\O_{K,v}^{n+1})\right)B^{|\w|}+ O\left(\epsilon_\Omega B^{|\w|-w_{\min}/d}\right).
\end{align}
Note that as $m_\infty$ is the product of measures $(m_v)_{v|\infty}$, one has
\[
\frac{m_\infty(\Omega_\infty^{\aff}\cap \mcF(1))}{m_\infty(\mcF(1))}=\prod_{v|\infty}\frac{m_v(\{x\in \Omega_v^{\aff} : \Ht_v(x)\leq 1\})}{m_v(\{x\in K_v^{n+1} : \Ht_v(x)\leq 1\})},
\]
and thus
\begin{align}\label{eq:VolOmegaF(1)}
m_\infty(\Omega_\infty^{\aff}\cap \mcF(1))=m_\infty(\mcF(1))\prod_{v|\infty}\frac{m_v(\{x\in \Omega_v^{\aff} : \Ht_v(x)\leq 1\})}{m_v(\{x\in K_v^{n+1} : \Ht_v(x)\leq 1\})}.
\end{align}
Computing the volume of $\mcF(1)$ as in \cite[Proposition 5.3]{Den98}, we obtain
\[
m_\infty(\mcF(1))=\left(2^{r_1+r_2}\pi^{r_2}\right)^{n+1}R_K|\w|^{r_1+r_2-1}.
\]
This differs by a factor of $2^{(n+1)r_2}$ from Deng's result since we are using the usual Haar measure on $K_\infty$, rather than the Lebesgue measure.

Observe that the leading cofficient of the asymptotic (\ref{eq:I<=c asymptotic}) equals $\zeta_K(|\w|)\kappa/h_K$, where $\kappa$ is as in the statement of Theorem \ref{thm:WProjFinFin}.

Note that for $x\in \c^\w$, $\mfI_\w(x)=\a\c$ for some ideal $\a\subseteq \O_K$. Then $\Ht_\w(x)=\Ht_{\w,\infty}(x)/N_{K/\Q}(\a\c)$. Therefore 
\[
M'(\Omega,\c,B)=\sum_{\a\subseteq \O_K} M(\Omega,\a\c, B/N_{K/\Q}(\a))=\sum_{\substack{\a\subseteq \O_K \\ N_{K/\Q}(\a)\leq B}} M(\Omega,\a\c, B/N_{K/\Q}(\a)),
\]
where the inequality $N_{K/\Q}(\a)\leq B$ comes from the fact that $M(\Omega,\c,B)=0$ if $B<1$, since all points have height greater than or equal to $1$.

Since the ideals of $\O_K$ form a poset (ordered by inclusion) we may use M\"obius inversion (see, e.g., \cite[\S 3.7]{Sta97} for details about M\"obius inversion for posets).
Applying M\"obius inversion, and using our asymptotic (\ref{eq:I<=c asymptotic}) for $M'(\Omega,\c,B)$, we have that
\begin{align*}
M(\Omega,\c,B) &=\sum_{\substack{\a\subseteq \O_K\\ N_{K/\Q}(\a)\leq B}} \mu(\a)\left(\frac{\zeta_K(|\w|)\kappa}{h_K}\left(\frac{B}{N_{K/\Q}(\a)}\right)^{|\w|}+ O\left(\epsilon_\Omega\left(\frac{B}{N_{K/\Q}(\a)}\right)^{|\w|-w_{\min}/d}\right)\right)\\
&=\frac{\zeta_K(|\w|)\kappa}{h_K} B^{|\w|}\left(\sum_{\a\subseteq \O_K} \mu(\a)\frac{1}{N_{K/\Q}(\a)^{|\w|}}
 - \sum_{\substack{\a\subseteq \O_K\\ N_{K/\Q}(\a)> B}}\mu(\a)\frac{1}{N_{K/\Q}(\a)^{|\w|}}\right)\\ 
&\hspace{20mm} +O\left(\epsilon_\Omega B^{|\w|-w_{\min}/d} \sum_{\substack{\a\subseteq \O_K\\ N_{K/\Q}(\a)\leq B}}\frac{1}{N_{K/\Q}(\a)^{|\w|-w_{\min}/d}}\right)\\
&=\frac{\zeta_K(|\w|)\kappa}{h_K} B^{|\w|}\left(\frac{1}{\zeta_K(|\w|)}
 - O\left(B^{-|\w|+1}\right)\right) 
+ \begin{cases}
O\left(\epsilon_\Omega B\log(B)\right) & {\substack{\text{ if } \w=(1,1)\\ \text{and}\ K=\Q,}}\\
O\left(\epsilon_\Omega B^{|\w|-w_{\min}/d}\right) & \text{ else, }
\end{cases}\\
&= \frac{\kappa}{h_K} B^{|\w|}
 + \begin{cases}
O\left(\epsilon_\Omega B\log(B)\right) & \text{ if } \w=(1,1) \text{ and } K=\Q,\\
O\left(\epsilon_\Omega B^{|\w|-w_{\min}/d}\right) & \text{ else. }
\end{cases}
\end{align*}
 Summing over the $h_K$ ideal class representatives $\c_i$, we arrive at Theorem \ref{thm:WProjFinFin}.
\end{proof}

\section{Counting elliptic curves}\label{sec:EC}

In this section we apply Theorem \ref{thm:WProjFinFin} to the moduli stack of elliptic curves.

\subsection{Counting elliptic curves with a prescribed local condition}\label{subsec:CountingLocalConditions}

 An application of Theorem \ref{thm:WProjFinFin} gives the following result:

\begin{corollary}\label{cor:EllipticCount}
Let $\mcE_K(B)$ denote the set of isomorphism classes of elliptic curves over $K$ with height less than $B$. Then
\[
\#\mcE_K(B)=\kappa B^{5/6}+O(B^{5/6-1/3d}),
\]
where
\[
\kappa=\frac{h_K (2^{r_1+r_2}\pi^{r_2})^{2} R_K 10^{r_1+r_2-1}\gcd(2,\varpi_K)}{\varpi_K |\Delta_K| \zeta_K(10)}.
\]
\end{corollary}

\begin{proof}
Recall that the moduli stack of elliptic curves, $\mcX_{\GL_2(\Z)}$, is isomorphic to the weighted projective stack $\mcP(4,6)$. Also recall that the naive height on $\mcX_{\GL_2(\Z)}$ corresponds to the twelfth power of the height on $\mcP(4,6)$ with respect to the tautological bundle, i.e., the height $\Ht_{(4,6)}$. Therefore, the desired result follows from applying Theorem \ref{thm:WProjFinFin} to $\mcP(4,6)$ with trivial local conditions (i.e., $S=\emptyset$).
\end{proof}

We now use Theorem \ref{thm:WProjFinFin} to count elliptic curves with a prescribed local condition. 

\vspace{2mm}

\begin{reptheorem}[\ref{thm:LocalConditions}]
Let $K$ be a degree $d$ number field, and let $\p\subset \O_K$ be a prime ideal of norm $q$ such that $2\nmid q$ and $3\nmid q$. Let $\msL$ be one of the local conditions listed in Table \ref{tab:LocalConditions}. Then the number of elliptic curves over $K$ with naive height less than $B$ and which satisfy the local condition $\msL$ at $\p$ is
\[
\kappa\kappa_\msL B^{5/6}+O\left(\epsilon_\msL B^{\frac{5}{6}-\frac{1}{3d}}\right),
\]
where
\[
\kappa=\frac{h_K \left(2^{r_1+r_2}\pi^{r_2}\right)^{2}R_K10^{r_1+r_2-1}\gcd(2,\varpi_K)}{\varpi_{K} |\Delta_K|\zeta_K(10)},
\]
and where $\kappa_\msL$ and $\epsilon_\msL$ are as in Table \ref{tab:LocalConditions}.
\end{reptheorem}

\begin{proof}
\underline{Good reduction case.} Note that the number of elliptic curves over $\F_q$ with good reduction is
\[
\#\{(a,b)\in \F_q\times\F_q : 4a^3+27b^2\not\equiv 0 \pmod{q}\}=q^2-q. 
\]
For each pair $(a,b)$ in the above set fix a lift in $\O_{K,v}^2$, which, by an abuse of notation, we will also denote $(a,b)$. Then consider the irreducible local condition
\[
\Omega_{\p,0}^{\aff}(a,b)\defeq\left\{(x_0,x_1)\in \O_{K,v}^{2} : |x_0-a|_v\leq \frac{1}{q},\ |x_1-b|_v\leq \frac{1}{q}\right\}.
\]
We find the $\p$-adic measure
\[
m_\p(\Omega_{\p,0}^{\aff}(a,b))=\frac{1}{q^2}.
\]
For $t\in \Z_{\geq 0}$, the sets
\[
\Omega_{\p,t}^{\aff}(a,b)=\pi_v^t\ast_{(4,6)}\Omega_{\p,0}^{\aff}(a,b)=\left\{(x_0,x_1)\in \O_{K,v}^{2} : |x_0-q^{4t}a|_v\leq \frac{1}{q^{4t+1}},\ |x_1-q^{6t}b|_v\leq \frac{1}{q^{6t+1}}\right\}
\]
each have $\p$-adic measure
\[
m_\p(\Omega_{\p,t}^{\aff}(a,b))=\frac{1}{q^{10t+2}}.
\]
Applying Theorem \ref{thm:WProjFinFin} to each $\Omega_{\p,0}^{\aff}(a,b)$, and summing over all the $q^2-q$ possible $(a,b)$ pairs, it follows that the number of elliptic curves of bounded height over $K$ with good reduction at $\p$ is 
\begin{align}\label{eq:good_reduction}
\kappa\kappa_\msL B^{5/6}+O\left(\epsilon B^{\frac{5}{6}-\frac{1}{3d}}\right),
\end{align}
where 
\[
\kappa_\msL=\sum_{(a,b)}\sum_{t=0}^\infty m_{\p}(\Omega_{\p,t}^{\aff}(a,b))=(q^2-q)\sum_{t=0}^\infty \frac{1}{q^{10t+2}}=(q^2-q)\frac{1}{q^2}\frac{q^{10}}{q^{10}-1}=
\frac{q-1}{q}\frac{q^{10}}{q^{10}-1},
\]
and
\[
\epsilon=\sum_{(a,b)}\sum_{t=0}^\infty\max\{q^{4t+1},q^{6t+1}\}m_{\p}(\Omega_{\p,t}^{\aff}(a,b))
=(q^2-q)\sum_{t=0}^\infty \frac{q^{6t+1}}{q^{10t+2}}
=\frac{q^5-q^4}{q^4-1}.
\]
As the error term can be rewritten as $O(q B^{\frac{5}{6}-\frac{1}{3d}})$, we may replace the coefficient $\epsilon$ in the asymptotic (\ref{eq:good_reduction}) with $\epsilon_\msL=q$.

\underline{Kodaira type III* case.} By Tate's algorithm \cite{Tat75}, an elliptic curve over $K_\p$ given in short Weierstrass form, $E:y^2=x^3+ax+b$, has type III* reduction if and only if $v_\p(a)=3$ and $v_\p(b)\geq 5$. Observe that
\[
\#\{(a,b)\in \O_K/\p^5\times \O_K/\p^5 : v_\p(a)=3,\ v_\p(b)\geq 5\}=q(q-1).
\]
For each pair $(a,b)$ in the above set fix a lift in $\O_{K,v}^2$, which, by an abuse of notation, we will also denote $(a,b)$. Then consider the irreducible local condition
\[
\Omega_{\p,0}^{\aff}(a,b)\defeq \left\{(x_0,x_1)\in \O_{K,v}^{2} : |x_0-a|_v\leq \frac{1}{q^5},\ |x_1-b|_v\leq \frac{1}{q^5}\right\}.
\]
We find that the $\p$-adic measure of this set is
\[
m_\p(\Omega_{\p,0}^{\aff}(a,b))=\frac{1}{q^{10}}.
\]
The sets
\[
\Omega_{\p,t}^{\aff}(a,b)= \pi_v^t\ast_{(4,6)}\Omega_{\p,0}^{\aff}(a,b)=\left\{(x_0,x_1)\in \O_{K,v}^{2} : |x_0-q^{4t}a|_v\leq \frac{1}{q^{4t+5}},\ |x_1-q^{6t}b|_v\leq \frac{1}{q^{6t+5}}\right\}
\]
each have $\p$-adic measure
\[
m_\p(\Omega_{\p,t}^{\aff}(a,b))=\frac{1}{q^{10t+10}}.
\]

Applying Theorem \ref{thm:WProjFinFin} to each $\Omega_{\p,0}^{\aff}(a,b)$, and summing over all pairs $(a,b)$, yields the following asymptotic for the number of elliptic curves of bounded height over $K$ with reduction of Kodaira type III* at $\p$:
\begin{align}\label{eq:Kodaira_type_III*}
\kappa\kappa_\msL B^{5/6}+O\left(\epsilon B^{\frac{5}{6}-\frac{1}{3d}}\right),
\end{align}
where 
\[
\kappa_\msL
=\sum_{(a,b)}\sum_{t=0}^\infty m_{\p}(\Omega_{\p,t}^{\aff}(a,b))=q(q-1)\sum_{t=0}^\infty \frac{1}{q^{10t+10}}=(q^2-q)\frac{1}{q^{10}}\frac{q^{10}}{q^{10}-1}=\frac{q-1}{q^9}\frac{q^{10}}{q^{10}-1},
\]
and
\[
\epsilon=\sum_{(a,b)}\sum_{t=0}^\infty\max\{q^{4t+5},q^{6t+5}\}m_{\p}(\Omega_{\p,t}^{\aff}(a,b))
=(q^2-q)\sum_{t=0}^\infty \frac{q^{6t+5}}{q^{10t+10}}
=\frac{q-1}{q^{4}-1}.
\]
As the error term can be rewritten as $O(q^{-3}B^{\frac{5}{6}-\frac{1}{3d}})$, we may replace the coefficient $\epsilon$ in the asymptotic (\ref{eq:Kodaira_type_III*}) with $\epsilon_\msL=1/q^{3}$.

The other cases can be proven similarly to those given above.
\end{proof}

We now turn to estimating the number of elliptic curves with a prescribed trace of Frobenius. For $b,c\in \F_q$, let $E_{b,c}$ denote the short Weierstrass model $y^2=x^3+bx+c$, with disciminant $\Delta(b,c)=-16(4b^3+27c^2)$ and trace of Frobenius $a_q(E_{b,c})$. We study the counting function
\begin{align}\label{eq:defH(a,q)}
H(a,q)\defeq\#\{(b,c)\in \F_q\times \F_q : \Delta(b,c)\neq 0,\ a_q(E_{b,c})=a\}.
\end{align}
This function is closely related to \emph{Hurwitz-Kronecker class numbers}, which we now recall. Let $K$ be an imaginary quadratic field with ring of integers $\O_K$, and let $\O$ be an order in $K$ with class number $h(\O)$. The \textbf{Hurwitz-Kronecker class number} of $\O$, denoted $H(\disc(\O))$, is defined as
\begin{align}\label{eq:Hurwitz-definition}
H(\disc(\O))\defeq\sum_{\O\subset \O' \subset \O_K} \frac{h(\O')}{\# \O'^\times}.
\end{align}
This definition of the Hurwitz-Kronecker class number agrees with the definition given in \cite{Len87}, but is half as large as it is sometimes defined (such as in \cite{Cox89}). Let $q=p^n$ be a power of a prime $p>3$. Then a beautiful result, following largely from work of Deuring \cite{Deu41}, expresses the number of elliptic curves over a finite field $\F_q$ with a prescribed trace of Frobenius in terms of Hurwitz-Kronecker class numbers:
\begin{equation}
\begin{aligned}\label{eq:Deuring}
H(a,q)=\begin{cases}
(q-1) H(a^2-4q) & \text{ if } |a|<2\sqrt{q} \text{ and } p\nmid a\\
(q-1)H(-4p) & \text{ if } n \text{ is odd and } a=0\\
\frac{q-1}{12}\left(p+6-4\legendre{-3}{p}-3\legendre{-4}{p}\right) & \text{ if } n \text{ is even and } a^2=4q\\
(q-1)\left(1-\legendre{-3}{p}\right) & \text{ if } n \text{ is even and } a^2=q\\
(q-1)\left(1-\legendre{-4}{p}\right) & \text{ if } n \text{ is even and } a=0\\
0 & \text{ otherwise.}
\end{cases}
\end{aligned}
\end{equation}
This formula can be derived from \cite[Theorem 4.6]{Sch87}, which counts the number of isomorphism classes of elliptic curves over $\F_q$ with a prescribed trace of Frobenius. From this, one can deduce formula (\ref{eq:Deuring}) by noting that the isomorphism class of each $E_{b,c}$ contains exactly $q-1$ elliptic curves. To see that each isomorphism class contains exactly $q-1$ elliptic curves, note that $E_{b,c}\cong E_{d,e}$ if and only if $(d,e)=(\lambda^4b, \lambda^6c)$ for some $\lambda\in \F_q^\times$, and since $p>3$ the pairs $(\lambda^4b, \lambda^6c)$ will be distinct for distinct $\lambda\in \F_q^\times$.
(See \cite[Theorem 14.18]{Cox89} for a proof of formula (\ref{eq:Deuring}) in the case that $q=p$ is a prime.)\\

\vspace{2mm}

\begin{reptheorem}[\ref{thm:aqLocalCondition}]
Let $K$ be a degree $d$ number field, and let $\p$ be a prime ideal of $\O_K$ of degree $n$  above a rational prime $p>3$. Set $q=p^n$. Let $a\in \Z$ be an integer satisfying $|a|\leq 2\sqrt{q}$. Then the number of elliptic curves over $K$ with naive height less than $B$, good reduction at $\p$, and which have trace of Frobenius $a$ at $\p$ is
\[
\kappa\kappa_{n,a} B^{5/6}+O\left(\epsilon_{n,a} B^{\frac{5}{6}-\frac{1}{3d}}\right),
\]
where
\[
\kappa=\frac{h_K \left(2^{r_1+r_2}\pi^{r_2}\right)^{2}R_K10^{r_1+r_2-1}\gcd(2,\varpi_K)}{\varpi_{K} |\Delta_K|\zeta_K(10)},
\]
and where $\kappa_{n,a}$ and $\epsilon_{n,a}$ are as in Table \ref{tab:aqLocalConditions}.
\end{reptheorem}

\begin{proof}

For each of the $H(a,q)$ pairs $(b,c)$ in the set
\[
\{(b,c)\in \F_q\times \F_q : \Delta(b,c)\neq 0,\ a_q(E_{b,c})=a\},
\]
fix a lift in $\O_{K,v}^2$, which, by an abuse of notation, we will also denote $(b,c)$. Then consider the irreducible local condition
\[
\Omega_{\p,0}^{\aff}(b,c)\defeq\left\{(x_0,x_1)\in \O_{K,v}^{2} : |x_0-b|_v\leq \frac{1}{q},\ |x_1-c|_v\leq \frac{1}{q}\right\}.
\]
We find that the $\p$-adic measure of each of these local conditions is
\[
m_\p(\Omega_{\p,0}^{\aff}(b,c))=\frac{1}{q^2}.
\]
More generally, for $t\in \Z_{\geq 0}$, the sets
\[
\Omega_{\p,t}^{\aff}(b,c)=\pi_v^t\ast_{(4,6)}\Omega_{\p,0}^{\aff}(b,c)=\left\{(x_0,x_1)\in \O_{K,v}^{2} : |x_0-q^{4t}b|_v\leq \frac{1}{q^{4t+1}},\ |x_1-q^{6t}c|_v\leq \frac{1}{q^{6t+1}}\right\}
\]
each have $\p$-adic measure
\[
m_\p(\Omega_{\p,t}^{\aff}(b,c))=\frac{1}{q^{10t+2}}.
\]
Applying Theorem \ref{thm:WProjFinFin} to each $\Omega_{\p,0}^{\aff}(b,c)$, and summing over all the $H(a,q)$ possible $(b,c)$ pairs, it follows that the number of elliptic curves of bounded height over $K$ with good reduction at $\p$ and trace of Frobenius $a$ at $\p$, is 
\begin{align}\label{eq:ap}
\kappa\kappa' B^{5/6}+O\left(\epsilon B^{\frac{5}{6}-\frac{1}{3d}}\right),
\end{align}
where 
\[
\kappa'=\sum_{(b,c)}\sum_{t=0}^\infty m_{\p}(\Omega_{\p,t}^{\aff}(b,c))=H(a,q)\sum_{t=0}^\infty \frac{1}{q^{10t+2}}=H(a,q)\frac{1}{q^2}\frac{q^{10}}{q^{10}-1},
\]
and
\[
\epsilon=\sum_{(b,c)}\sum_{t=0}^\infty\max\{q^{4t+1},q^{6t+1}\}m_{\p}(\Omega_{\p,t}^{\aff}(b,c))
=H(a,q)\sum_{t=0}^\infty \frac{q^{6t+1}}{q^{10t+2}}
=H(a,q)\frac{q^3}{q^4-1}.
\]
From this the desired asymptotics can now be derived using (\ref{eq:Deuring}).

For example, in the case that $a$ is an integer satisfying $|a|< 2\sqrt{q}$ and $p\nmid a$, we have that $H(a,q)=(q-1)H(a^2-4q)$. Thus
\[
\kappa'=\frac{H(a,q)}{q^2}\frac{q^{10}}{q^{10}-1}=\frac{(q-1)H(a^2-4q)}{q^2}\frac{q^{10}}{q^{10}-1}=\kappa_{n,a}
\]
and
\[
\epsilon=(q-1)H(a^2-4q)\frac{q^3}{q^4-1}.
\]
As the error term can be rewritten as $O\left(H(a^2-4q)B^{\frac{5}{6}-\frac{1}{3d}}\right)$, we may replace the coefficient $\epsilon$ in the asymptotic (\ref{eq:ap}) with $\epsilon_{n,a}=H(a^2-4q)$.

\end{proof}

\section{Average analytic ranks}\label{sec:AvgRanks}

In this section we use our results from the previous section on counting elliptic curves with prescribed local conditions in order to give a bound for the average analytic rank of elliptic curves over number fields. Our strategy will mainly follow that of Cho and Jeong \cite{CJ23a, CJ23b}, and is related to Brumer's strategy (see Remark \ref{rem:brumer}).

\subsection{L-functions of elliptic curves}
In this subsection we recall basic facts about $L$-functions of elliptic curves. Our exposition mostly follows \cite[Appendix A]{Mil02}.

Let $E$ be an elliptic curve of discriminant $\Delta_E$ over a number field $K$. For each finite place $v\in \Val_0(K)$, let $E_v$ denote the reduction of $E$ at $v$, let $q_v$ be the order of the residue field $k_v$ of $K$ at $v$, 
let $N_v=\# E_v(\F_{q_v})$, and let $a_v=q_v+1-N_v$. If $E$ has bad reduction at $v$, set
\[
b_v\defeq \begin{cases}
1 & \text{ if $E$ has split multiplicative reduction at $v$},\\
-1 & \text{ if $E$ has nonsplit multiplicative reduction at $v$},\\
0 & \text{ if $E$ has additive reduction at $v$}.
\end{cases}
\]

For each $v\in \Val_0(K)$, define a polynomial $L_v(E/K,T)$ by
\[
L_v(E/K,T)\defeq\begin{cases}
1-a_v T+ q_v T^2 & \text{ if $E$ has good reduction at $v$},\\
1-b_vT & \text{ if $E$ has bad reduction at $v$}.
\end{cases}
\]
The \textbf{Hasse--Weil $L$-function} $L(E/K,s)$ of $E$ over $K$ is defined as the Euler product
\[
L(E/K,s)\defeq \prod_{v\in \Val_0(K)} L_v(E/K, q_v^{-s})^{-1}.
\]
The logarithmic derivative of $L(E/K,s)$ is
\[
\frac{L'}{L}(E/K,s)=-\sum_{\substack{v\in \Val_0(K)\\ v(\Delta_E)>0}} \frac{d}{ds}\left(\log(1-b_v q_v^{-s})\right)-\sum_{\substack{v\in \Val_0(K)\\ v(\Delta_E)=0}} \frac{d}{ds}\left(\log(1-a_v q_v^{-s}+q_v^{1-2s})\right).
\]
The first sum can be rewritten as
\[
-\sum_{\substack{v\in \Val_0(K)\\ v(\Delta_E)>0}} \frac{d}{ds}\left(\log(1-b_v q_v^{-s})\right)=\sum_{\substack{v\in \Val_0(K)\\ v(\Delta_E)>0}} \frac{d}{ds}\left(\sum_{k=1}^\infty \frac{b_v^k}{k q_v^{ks}}\right)=-\sum_{\substack{v\in \Val_0(K)\\ v(\Delta_E)>0}} \sum_{k=1}^\infty \frac{b_v^k \log(q_v)}{q_v^{ks}}.
\]
For the second sum, factor $1-a_v T+q_v T^2$ as $(1-\alpha_v T)(1-\beta_v T)$, where $\alpha_v+\beta_v=a_v$ and $\alpha_v\beta_v=q_v$. Then
\[
-\sum_{\substack{v\in \Val_0(K)\\ v(\Delta_E)=0}} \frac{d}{ds}\left(\log(1-a_v q_v^{-s}+q_v^{1-2s})\right)=-\sum_{\substack{v\in \Val_0(K)\\ v(\Delta_E)=0}} \sum_{k=1}^\infty \frac{(\alpha_v^k+\beta_v^k) \log(q_v)}{q_v^{ks}}.
\]
We thus obtain the following expression for the logarithmic derivative of $L(E/K,s)$:
\begin{align*}
\frac{L'}{L}(E/K,s)=-\sum_{\substack{v\in \Val_0(K)\\ v(\Delta_E)>0}} \sum_{k=1}^\infty \frac{b_v^k \log(q_v)}{q_v^{ks}}-\sum_{\substack{v\in \Val_0(K)\\ v(\Delta_E)=0}} \sum_{k=1}^\infty \frac{(\alpha_v^k+\beta_v^k) \log(q_v)}{q_v^{ks}}.
\end{align*}
Define the \textbf{von Mangoldt function} on integral ideals $\a\subseteq \O_K$ as
\[
\Lambda_K(\a)\defeq \begin{cases}
\log(q_v) & \text{ if } \a=\p_v^k \text{ for some } k,\\
0 & \text{ otherwise.}
\end{cases}
\]
For prime powers, set
\[
\widehat{a_E}(\p_v^k)=\begin{cases}
\alpha_v^k+\beta_v^k & \text{ if $E$ has good reduction at $\p_v$,}\\
b_v^k &  \text{ if $E$ has bad reduction at $\p_v$,}
\end{cases}
\]
and for $\a$ not a power of a prime set $\widehat{a_E}(\a)=0$. Then our expression for the logarithmic derivative of $L(E/K,s)$ becomes
\begin{align}\label{eq:L-logarithmic_derivative}
\frac{L'}{L}(E/K,s)=-\sum_{\a\subseteq \O_K} \frac{\widehat{a_E}(\a)\Lambda_K(\a)}{N_{K/\Q}(\a)^s},
\end{align}
where the sum is over the non-zero integral ideals of $\O_K$.

Let $\Gamma(s)\defeq\int_0^\infty t^{s-1}e^{-t} dt$ be the usual gamma function. Let $\Gamma_K(s)$ be the gamma factor
\[
\Gamma_K(s)\defeq\left((2\pi)^{-s} \Gamma(s)\right)^{[K:\Q]}.
\]
Let $\mathfrak{f}_{E/K}$ be the conductor of $E$ over $K$, and define a constant
\[
A_{E/K}\defeq N_{K/\Q}(\mathfrak{f}_{E/K}) \Delta_K^2.
\]
Define the \textbf{completed Hasse--Weil $L$-function}, $\Lambda(E/K, s)$, as
\[
\Lambda(E/K,s)\defeq A_{E/K}^{s/2}\Gamma_K(s) L(E/K,s).
\]

We now assume that our elliptic curve $E$ is modular over $K$, so that the $L$-function $L(E/K,s)$ is automorphic. This implies that the Hasse--Weil conjecture holds true for $E$, and thus
\[
\Lambda(E/K,s)=\epsilon(E) \Lambda(E/K, 2-s),
\]
where $\epsilon(E)\in\{\pm 1\}$ is the root number. We see that the logarithmic derivative of $\Lambda(E/K, s)$ is
\begin{align}\label{eq:Lambda-logarithmic_derivative}
\frac{\Lambda'}{\Lambda}(E/K, s)=\frac{\log(A_{E/K})}{2}+\frac{\Gamma_K'}{\Gamma_K}(s)+\frac{L'}{L}(E/K,s)=-\frac{\Lambda'}{\Lambda}(E/K,2-s).
\end{align}

We now state the explicit formula for $L$-functions of elliptic curves over number fields:

\begin{proposition}[Explicit Formula]\label{prop:explicit-formula}
Let $\phi\colon\C\to\C$ be a partial function that is even and analytic in the strip $|\Im(s)|\leq \frac{1}{2}+\epsilon$ for some $\epsilon>0$, and assume that on this strip $|\phi(s)|\ll (1+|s|)^{-(1+\delta)}$ for some $\delta>0$ as $|\RE(s)|\to\infty$. For $x\in \R$ assume that $\phi(x)\in \R$, and set 
\[
\widehat{\phi}(t)\defeq\int_\R \phi(x) e^{-2\pi i tx} dx,
\]
the Fourier transform.
Let $E$ be a modular elliptic curve over a number field $K$. Suppose that the $L$-function $L(E/K,s)$ associated to $E$ satisfies the Generalized Riemann Hypothesis, so that all non-trivial zeros $\rho_j$ of $L(E/K,s)$ are of the form
\[
\rho_j=1+i\gamma_j \text{ for some } \gamma_j\in \R.
\]
Let $r(E)\defeq\ord_{s=1}(L(E/K,s))$ be the analytic rank of $E$. Then we have
\begin{align*}
&r(E)\phi(0)+\sum_{\rho_j\neq 1} \phi(\gamma_j)\\
&\hspace{5mm} =\frac{\widehat{\phi}(0)}{2\pi}\log(A_{E/K})+\frac{1}{\pi}\int_{\R}\frac{\Gamma_K'}{\Gamma_K}(1+iy) \phi(y) dy 
-\frac{1}{\pi} \sum_{\a\subseteq \O_K} \frac{ \widehat{a_E}(\a) \Lambda_K(\a)}{N_{K/\Q}(\a)}\cdot\widehat{\phi}\left(\frac{\log(N_{K/\Q}(\a))}{2\pi}\right),
\end{align*}
where the sum on the left hand side is over the zeros $\rho_j\neq 1$ with relevant multiplicity.
\end{proposition}

Proposition \ref{prop:explicit-formula} is a slight generalization of \cite[Theorem 5.12]{IK04} (see also \cite{Mes86} and \cite[Appendix A]{Mil02} for the specific case of $L$-functions of elliptic curves). In \cite{IK04}, the condition on the test function $\phi$ is that it is a Schwartz function. However, as observed in \cite[Lemma 1]{GG07}, the proof of \cite[Theorem 5.12]{IK04} can be adapted to work for the set of functions $\phi$ satisfying the conditions in Proposition \ref{prop:explicit-formula}. We will use the following modified version of the explicit formula:

\begin{corollary}[Modified Explicit Formula]\label{cor:modified-expicit-formula}
Maintaining the notation from Proposition \ref{prop:explicit-formula} and letting $X$ be a parameter, we have that
\begin{align}
\begin{split}
&r(E)\phi(0)+\sum_{\gamma_j\neq 0} \phi\left(\gamma_j\frac{\log(X)}{2\pi}\right) \label{eq:modified-explicit-formula}\\
&\hspace{1mm}= \frac{\widehat{\phi}(0) \log(A_{E/K})}{\log(X)} -\frac{2}{\log(X)} \sum_{\a\subseteq \O_K} \frac{ \widehat{a_E}(\a) \Lambda_K(\a)}{N_{K/\Q}(\a)}\cdot\widehat{\phi}\left(\frac{\log(N_{K/\Q}(\a))}{\log(X)}\right)+O\left(\frac{1}{\log(X)}\right).
\end{split}
\end{align}
\end{corollary}

\begin{proof}
From the explicit formula (Proposition \ref{prop:explicit-formula}) it follows that
\begin{align}
r(E)\phi(0)+\sum_{\gamma_j\neq 0} \phi\left(\gamma_j\frac{\log(X)}{2\pi}\right)
&= \frac{\widehat{\phi}(0) \log(A_{E/K})}{\log(X)}+\frac{1}{\pi}\int_{\R}\frac{\Gamma_K'}{\Gamma_K}(1+iy) \phi\left(y\frac{\log(X)}{2\pi}\right) dy \nonumber\\
& -\frac{2}{\log(X)} \sum_{\a\subseteq \O_K} \frac{ \widehat{a_E}(\a) \Lambda_K(\a)}{N_{K/\Q}(\a)}\cdot\widehat{\phi}\left(\frac{\log(N_{K/\Q}(\a))}{\log(X)}\right). \label{eq:explicit-formulaX}
\end{align}

A classical estimate for the logarithmic derivative of the gamma function is
\[
\left| \frac{\Gamma'}{\Gamma}(1+iy)\right|=O(\log(|y|+2)).
\]
From this, and the fact that $\phi$ is absolutely integrable (since $|\phi(s)|\ll (1+|s|)^{-(1+\delta)}$), it follows that
\[
\int_{\R}\frac{\Gamma_K'}{\Gamma_K}(1+iy) \phi\left(y\frac{\log(X)}{2\pi}\right) dy=O\left(\frac{1}{\log(X)}\right).
\]
Substituting this into (\ref{eq:explicit-formulaX}) gives the modified explicit formula.
\end{proof}

We now use the modified explicit formula to obtain an expression that bounds the analytic rank.
Let $E$ be an elliptic curve with minimal discriminant $\mfD_{E/K}$ and of height less than or equal to $X$. For each such $E$ we have that $N_{K/\Q}(\mfD_{E/K})\ll X$, and therefore
\[
\log\left(N_{K/\Q}(\mfD_{E/K})\right)\leq \log(X)+O(1).
\]
 But since the conductor $\f_{E/K}$ divides the minimal discriminant $\mfD_{E/K}$, one has $N_{K/\Q}(\f_{E/K})\leq N_{K/\Q}(\mfD_{E/K})$, and thus
\[
\log(N_{K/\Q}(\f_{E/K}))/\log(X)\leq 1+O(1/\log X).
\]
 It follows that $\log(A_{E/K})/\log(X)\leq 1+O(1/\log X)$. 
By Corollary \ref{cor:modified-expicit-formula}, this observation gives the inequality
\begin{align*}
&r(E)\phi(0)+\sum_{\gamma_j\neq 0} \phi\left(\gamma_j\frac{\log(X)}{2\pi}\right)\\
&\hspace{1cm}\leq \widehat{\phi}(0)-\frac{2}{\log(X)} \sum_{\a\subseteq \O_K} \frac{ \widehat{a_E}(\a) \Lambda_K(\a)}{N_{K/\Q}(\a)}\cdot\widehat{\phi}\left(\frac{\log(N_{K/\Q}(\a))}{\log(X)}\right)
+O\left(\frac{1}{\log X}\right).
\end{align*}

We further simplify by rewriting the sum on the right hand side as
\[
\sum_{\a\subseteq \O_K} \frac{ \widehat{a_E}(\a) \Lambda_K(\a)}{N_{K/\Q}(\a)}\cdot\widehat{\phi}\left(\frac{\log(N_{K/\Q}(\a))}{\log(X)}\right)
=\sum_{v\in \Val_0(K)} \sum_{k\geq 1} \frac{ \widehat{a_E}(\p_v^k) \log(q_v)}{q_v^k}\cdot\widehat{\phi}\left(\frac{k \log(q_v)}{\log(X)}\right).
\]
By the Weil Conjectures for Elliptic curves, $|\alpha_v|=|\beta_v|=\sqrt{q_v}$ (see, e.g., \cite[\S V Theorem 2.3.1(a)]{Sil09}). From this, and the fact that $|b_v|\leq 1$, we have $|\widehat{a_E}(\p_v^k)|\leq 2q_v^{k/2}$. Using this, we bound the $k\geq 3$ part of the sum:
\begin{align*}
\sum_{v\in \Val_0(K)} \sum_{k\geq 3} \frac{ \widehat{a_E}(\p_v^k) \log(q_v)}{q_v^k}\cdot\widehat{\phi}\left(\frac{k \log(q_v)}{\log(X)}\right)
& \leq \sum_{v\in \Val_0(K)} \sum_{k\geq 3} \frac{ 2 \log(q_v)}{q_v^{k/2}} ||\widehat{\phi}||_\infty\\
&\leq \sum_{v\in \Val_0(K)} \frac{ 2 \log(q_v)}{q_v^{3/2} (1-q_v^{-1/2})} ||\widehat{\phi}||_\infty\\
&=O(1),
\end{align*}
where $||f||_\infty$ is the usual $L^\infty$-norm defined as
\[
||f||_\infty\defeq\inf\left\{a\geq 0 : m_L(\{x : |f(x)|>a\})=0\right\},
\]
where $m_L$ is the Lebesgue measure on $\R$. Also note that for any finite set of places $S\subset \Val_0(K)$  we clearly have 
\[
\sum_{v\in S} \sum_{k=1}^2 \frac{ \widehat{a_E}(\p_v^k) \log(q_v)}{q_v^k}\cdot\widehat{\phi}\left(\frac{k \log(q_v)}{\log(X)}\right)=O(1).
\]
Therefore, setting
\[
U_{k}(E,\phi, X)\defeq \sum_{\substack{v\in \Val_0(K) \\ 2\nmid q_v,\ 3\nmid q_v}} \frac{\widehat{a_E}(\p_v^k)\log(q_v)}{q_v^k} \cdot \widehat{\phi}\left( \frac{k\log(q_v)}{\log(X)}\right),
\]
we have
\[
r(E)\phi(0)+\sum_{\gamma_j\neq 0} \phi\left(\gamma_j\frac{\log(X)}{2\pi}\right)
\leq \widehat{\phi}(0)-\frac{2}{\log(X)} ( U_1(E,\phi ,X)+U_2(E,\phi,X))
+O\left(\frac{1}{\log X}\right).
\]

We now average over isomorphism classes of elliptic curves. For each elliptic curve $E$, let $\rho_{E,j}=1+i\gamma_{E,j}$ be the non-trivial zeros of the L-function $L(E/K,s)$. Let $\mcE_K(B)$ denote the set of isomorphism classes of elliptic curves over $K$ with height bounded by $B$. Setting
\begin{align*}
S_k(\phi,B)
&\defeq \frac{2}{\log(B) \#\mcE_{K}(B)} \sum_{E\in \mcE_{K}(B)} U_k(E,\phi,B)\\
&=\frac{2}{\log(B) \#\mcE_{K}(B)} \sum_{\substack{v\in \Val_0(K) \\ 2\nmid q_v,\ 3\nmid q_v}} \frac{\log(q_v)}{q_v^k} \cdot \widehat{\phi}\left( \frac{k\log(q_v)}{\log(B)}\right)\sum_{E\in\mcE_{K}(B)} \widehat{a_E}(\p_v^k),
\end{align*}
we have that
\begin{align}\label{eq:explicit-average-rank}
\begin{split}
\phi(0) \frac{1}{\#\mathcal{E}_{K}(B)} \sum_{E\in \mcE_{K}(B)} r(E)+\frac{1}{\# \mcE_{K}(B)} \sum_{E\in\mcE_{K}(B)}\sum_{\gamma_{E,j}\neq 0}\phi\left(\gamma_{E,j}\frac{\log(B)}{2\pi}\right)\\
\hspace{2cm} \leq \widehat{\phi}(0)-S_1(\phi,B)-S_2(\phi,B)+O\left(\frac{1}{\log(B)}\right).
\end{split}
\end{align}

\begin{remark}\label{rem:brumer}
So far we have followed the same strategy as Brumer \cite{Bru92}. The main difficulty in estimating the $S_k(\phi,B)$ is estimating the sums of the $\widehat{a_E}(\p_v^k)$. Brumer does this by relating the sums to certain exponential sums. Following Cho and Jeong \cite{CJ23a, CJ23b}, we will take a slightly different approach, instead estimating the sums using our estimates for the number of elliptic curves with a prescribed trace of Frobenius (Theorem \ref{thm:aqLocalCondition}).
\end{remark}

For our applications, we shall use the test function
\[
\phi(y)=\begin{cases}
\left(\frac{\sin(\pi \nu y)}{2\pi y}\right)^2 & \text{ if } y\neq 0,\\
\nu^2/4 & \text{ if }y=0,
\end{cases}
\]
whose Fourier transform is
\[
\widehat{\phi}(t)=\begin{cases}
\frac{1}{2}\left(\frac{\nu}{2}-\frac{|t|}{2}\right) &\text{ for } |t|\leq \nu\\
0 &\text{ for } |t|>\nu.
\end{cases}
\]
Observe that this test function satisfies the conditions of Proposition \ref{prop:explicit-formula} (despite not being a Schwartz function).

Our strategy will be to show that 
\[
-S_1(\phi,B)-S_2(\phi,B)=\frac{\phi(0)}{2}+o(1).
\]
This will imply the following upper bound for the average analytic rank of elliptic curves over the number field $K$:
\[
\frac{\widehat{\phi}(0)}{\phi(0)}+\frac{1}{2} = \frac{1}{\nu}+\frac{1}{2}.
\]

\subsection{Class number estimates}

In this subsection we give some estimates for the counting functions $H(a,q)$, which were defined in Subsection \ref{subsec:CountingLocalConditions} equation (\ref{eq:defH(a,q)}). These results will be used in the next subsection. 

\begin{proposition}\label{prop:H(a,q)-bound}
For any $a\in \Z$ we have
\[
H(a,q)\ll_{\epsilon} q^{3/2+\epsilon},
\]
as a function of $q$ for any $\epsilon>0$.
\end{proposition}

\begin{proof}
By formula (\ref{eq:Deuring}) for $H(a,q)$, it suffices to show $H(a^2-4q)\ll_{\epsilon} q^{1/2+\epsilon}$ and $H(-4p)\ll_{\epsilon} q^{1/2+\epsilon}$, and for this it will suffice to show that for all possible $n$ we have $H(n)\ll_{\epsilon} |n|^{1/2+\epsilon}$. As $n$ must be the discriminant of an order in an imaginary quadratic field, it must be negative, non-square, and congruent to $0$ or $1$ modulo $4$.

Let $\sqf(n)$ denote the squarefree part of $n$ (e.g., $\sqf(12)=3$), and let $\sq(n)$ be such that $n=\sqf(n)\cdot \sq(n)^2$ (e.g., $\sq(12)=2$). For non-square negative integers $D$ that are congruent to $0$ or $1$ modulo $4$, let $\O_{D}$ denote the imaginary quadratic order of discriminant $D$. We can rewrite equation (\ref{eq:Hurwitz-definition}) as
\[
H(n)=\sum_{\substack{d|\sq(n)\\ n/d^2\equiv 0,1\pmod{4}}} \frac{h(\O_{n/d^2})}{\#\O_{n/d^2}^{\times}}.
\]
Define a quadratic character
\begin{align*}
\chi_D:\Z&\to \C^{\times}\\
m & \mapsto \legendre{D}{m}.
\end{align*}
Observe that
\begin{align}\label{eq:limit}
\lim_{s\to 1}(s-1)\zeta_{\O_D}(s) 
= \sum_{m=1}^\infty \frac{\chi_D(m)}{m} 
=\sum_{m=1}^\infty \chi_D(m) \int_m^{\infty}\frac{1}{x^2}dx.
\end{align}
For any positive integer $r\in \Z_{\geq 1}$, the P\'{o}lya--Vinogradov inequality gives 
\[
\left|\sum_{m=1}^r \chi_D(m)\right|\leq 2\sqrt{|D|} \log(|D|).
\]
In particular, the functions $f_r(x)=\sum_{m=1}^r \chi_D(m)/x^2$ are dominated by $2\sqrt{|D|}\log(|D|)/x^{2}$ on $\R_{\geq 1}$.
Therefore, by the dominated convergence theorem, we may switch the sum and integral in (\ref{eq:limit}), so that
\begin{align*}
\sum_{m=1}^\infty \chi_D(m) \int_m^{\infty}\frac{1}{x^2}dx=\int_1^{\infty} \left(\sum_{m=1}^x \chi_D(m)\right) \frac{1}{x^2} dx.
\end{align*}
Splitting this integral and again applying the P\'{o}lya--Vinogradov inequality, we find that for any $\epsilon>0$
\begin{align*}
\lim_{s\to 1}(s-1)\zeta_{\O_D}(s) 
&= \int_1^{\sqrt{|D|}} \left(\sum_{m=1}^x \chi_D(m)\right) \frac{1}{x^2} dx 
+ \int_{\sqrt{|D|}}^{\infty} \left(\sum_{m=1}^x \chi_D(m)\right) \frac{1}{x^2} dx \\
&\ll  \int_1^{\sqrt{|D|}}  \frac{x}{x^2} dx 
+ \int_{\sqrt{|D|}}^{\infty}  \frac{\sqrt{|D|}\log(|D|)}{x^2} dx    \\
&\ll \log(|D|) + \log(|D|)  \\
&\ll |D|^{\epsilon}.
\end{align*}
From this and the analytic class number formula \cite[Theorem 1.1]{JP20}, 
we have the following upper bound for class numbers of orders in imaginary quadratic fields: 
\[
h(\O_D)\ll  |D|^{1/2+\epsilon}.
\]
Let $\tau(n)$ denote the number of divisors of $n$, and recall that $\tau(n)\ll_{\epsilon} |n|^{\epsilon}$ for any $\epsilon>0$. Combining the above observations, we have 
\[
H(n)\leq \sum_{\substack{d|\sq(n)\\ n/d^2\equiv 0,1\pmod{4}}} h(\O_{n/d^2})\ll \tau(\sq(n)) |n|^{1/2+\epsilon}\ll_{\epsilon} |n|^{1/2+\epsilon}.
\]
\end{proof}

The next proposition gives estimates for certain sums involving $H(a,q)$.

\begin{proposition}\label{prop:ClassNumberSums}
Let $q$ be a power of a prime $p>3$. For any $\epsilon>0$ we have the following estimates:
\begin{align*}
&\sum_{|a|\leq 2\sqrt{q}} H(a,q)= q^2-q,\\
&\sum_{|a|\leq 2\sqrt{q}} aH(a,q)=0,\\
&\sum_{|a|\leq 2\sqrt{q}} a^2 H(a,q)=q^3+O\left(q^{\frac{5}{2}+\epsilon}\right).
\end{align*}
\end{proposition}

\begin{proof}
\underline{The first sum:} We have that
\[
\sum_{|a|\leq 2\sqrt{q}} H(a,q)=\#\{(b,c)\in \F_q\times\F_q : \Delta(b,c)\neq 0\}=q^2-q.
\]

\underline{The second sum:} Pairing positive and negative terms, we have that
\[
\sum_{|a|\leq 2\sqrt{q}} a H(a,q)=\sum_{0<a \leq 2\sqrt{q}} \left(a H(a,q)-a H(a,q)\right)= 0,
\]
where we have used that $H(a,q)=H(-a,q)$.
 
\underline{The third sum:} This sum is the most subtle. As observed in Subsection \ref{subsec:CountingLocalConditions}, each isomorphism class of elliptic curves over $\F_q$ contains exactly $q-1$ elliptic curves, which can be thought of as the size of the orbit of an elliptic curve $E$ with respect to the weighted $\G_m(\F_q)$ action, with weights $4$ and $6$. Let $\mcE_{\F_q}$ denote the set of isomorphism classes of elliptic curves over $\F_q$. By the orbit stabilizer theorem
\[
H(a,q)=\sum_{\substack{E\in \mcE_{\F_q}\\ a_q(E)=a}}\frac{q-1}{\#\Aut_{\F_q}(E)}.
\]
 Therefore we can rewrite the third sum as
\[
\sum_{|a|\leq 2\sqrt{q}} a^2 H(a,q)
=(q-1)\sum_{|a|\leq 2\sqrt{q}}\sum_{{\substack{E\in \mcE_{\F_q}\\ a_q(E)=a}}} \frac{a^2}{\#\Aut_{\F_q}(E)}
=(q-1)\sum_{E\in \mcE_{\F_q}} \frac{a_q(E)^2}{\#\Aut_{\F_q}(E)}.
\] 
The final sum can be estimated using work of Birch \cite{Bir68} and Ihara \cite{Iha67}, as we now explain.
Write $q=p^r$. Define an indicator function
\[
\delta_{2\Z}(r)=\begin{cases}
1 & \text{ if $r$ is even}\\
0 & \text{ if $r$ is odd}.
\end{cases}
\] 
Let $T_k(m)$ be the $m$-th Hecke operator on the space of weight $k$ cusp forms of level 1, and let $\Tr(T_k(m))$ denote the trace of $T_k(m)$. 

The following result is due to Birch (in the $r=1$ case) \cite{Bir68} and Ihara (for $r>1$) \cite{Iha67} (see \cite[Theorem 2]{KP17}):

\begin{theorem}[Birch-Ihara]\label{thm:Ihara}
With notation as above, we have that
\begin{align*}
\frac{1}{q}\sum_{E\in \mcE_{\F_q}} \frac{a_q(E)^2}{\#\Aut_{\F_q}(E)}
\end{align*}
equals
\begin{align*}
&\frac{1}{q}\left(p^3\Tr(T_4(p^{r-2}))-\Tr(T_4(q)) +\frac{1}{2}\left( p^3 \sum_{i=0}^{r-2} \min\{p^i,p^{r-2-i}\}^3-\sum_{i=0}^{r} \min\{p^i, p^{r-i}\}^3\right)\right)\\
&+ p \Tr(T_2(p^{r-2})) - \Tr(T_2(q)) + \frac{1}{2}\left(p \sum_{i=0}^{r-2} \min\{p^i,p^{r-2-i}\} - \sum_{i=0}^r \min\{p^i,p^{r-i}\} \right) \\
&  + \sum_{i=0}^r p^i-p\sum_{i=0}^{r-2}p^i.
\end{align*}

\end{theorem}


We now show that for any $\epsilon>0$
\begin{align*}
\sum_{E\in \mcE_{\F_q}} \frac{a_q(E)^2}{\#\Aut_{\F_q}(E)}
 = q^2 + O(q^{3/2+\epsilon}).
\end{align*}
Using the Ramanujan--Peterson--Deligne bound for Hecke eigenvalues \cite{Del74}, one can show that for any prime power $p^s$ and any $\epsilon>0$,
\[
\Tr(T_k(p^s))\ll_k p^{s(k-1)/2}s\ll_k p^{s((k-1)/2+\epsilon)}
\] 
(see, e.g., the introduction of \cite{Pet18}).
In particular, 
\begin{align*}
p^3\Tr(T_4(p^{r-2}))-\Tr(T_4(q)) \ll  p^3 p^{(r-2)(3/2+\epsilon)} + q^{3/2+\epsilon} \ll q^{3/2+\epsilon}
\end{align*}
and
\begin{align*}
p \Tr(T_2(p^{r-2})) - \Tr(T_2(q))\ll p p^{(r-2)(1/2+\epsilon)}+ q^{1/2+\epsilon} \ll q^{1/2+\epsilon}.
\end{align*}

We now estimate the sums in Theorem \ref{thm:Ihara}. 
As a function of $q$, the sum $\sum_{i=0}^{r} \min\{p^i, p^{r-i}\}^3$ has on the order of $\log(q)$ terms, and each term is bounded above by $\left(p^{r/2}\right)^3=q^{3/2}$. Therefore 
\begin{align*}
\sum_{i=0}^{r} \min\{p^i, p^{r-i}\}^3\ll q^{3/2}\log(q)\ll q^{3/2+\epsilon}.
\end{align*}
Similar reasoning shows that 
\begin{align*}
p^3 \sum_{i=0}^{r-2} \min\{p^i,p^{r-2-i}\}^3\ll p^3\log(q)p^{3(r-2)/2}\ll q^{3/2+\epsilon}.
\end{align*}
We also compute 
\begin{align*}
p \sum_{i=0}^{r-2} \min\{p^i,p^{r-2-i}\} - \sum_{i=0}^r \min\{p^i,p^{r-i}\}
=-\sum_{i\in\{0,r\}} \min\{p^i,p^{r-i}\}
=-2,
\end{align*}
and 
\begin{align*}
\sum_{i=0}^r p^i-p\sum_{i=0}^{r-2}p^i=\sum_{i\in \{0,r\}} p^i = 1 + q.
\end{align*}
%
Combining the above estimates we find that
\begin{align*}
\frac{1}{q}\sum_{E\in \mcE_{\F_q}} \frac{a_q(E)^2}{\#\Aut_{\F_q}(E)}
&= \O\left(\frac{q^{3/2+\epsilon}}{q}\right) + O(q^{1/2+\epsilon})-\frac{2}{2}+(1+q)= q + O(q^{1/2+\epsilon}).
\end{align*}
From this we obtain
\[
\sum_{|a|\leq 2\sqrt{q}} a^2 H(a,q)=(q-1)\left(q^2 + O(q^{3/2+\epsilon})\right)=q^3+O(q^{5/2+\epsilon}).
\]
\end{proof}

\subsection{Estimating $S_1(\phi,B)$ and $S_2(\phi,B)$}

\begin{lemma}\label{lem:S1}
Let $\p\subset \O_K$ be a prime ideal of norm $N_{K/\Q}(\p)=q$ such that $2\nmid q$ and $3\nmid q$. Then, for any $\epsilon>0$, we have the following upper bound:
\[
\sum_{E\in \mcE_K(B)} \widehat{a_E}(\p)\ll_{\epsilon} \frac{1}{q} B^{5/6}+q^{3/2+\epsilon} B^{5/6-1/3d}.
\]
\end{lemma}

\begin{proof}
Let $\mcE_K^{\mult}(B)$ denote the set of elliptic curves over $K$ of naive height bounded by $B$ and multiplicative reduction at $\p$. Then we have
\[
\sum_{E\in \mcE_K(B)}\widehat{a_E}(\p)=\sum_{|a|\leq 2\sqrt{q}} \sum_{\substack{E\in\mcE_K(B)\\ a_\p(E)=a}} \widehat{a_E}(\p) + \sum_{E\in \mcE_K^{\mult}(B)} \widehat{a_E}(\p).
\]
By Theorem \ref{thm:LocalConditions}, we have that
\begin{align}\label{eq:mult-bound-1}
\left|\sum_{E\in \mcE_K^{\mult}(B)} \widehat{a_E}(\p)\right|\ll \#\mcE_K^{\mult}(B) \ll \frac{1}{q}B^{5/6}+B^{5/6-1/3d}.
\end{align}
By Proposition \ref{prop:ClassNumberSums}, Proposition \ref{prop:H(a,q)-bound}, and the proof of Theorem \ref{thm:aqLocalCondition}, we have that
\begin{align*}
\sum_{|a|\leq 2\sqrt{q}} \sum_{\substack{E\in\mcE_K(B)\\ a_\p(E)=a}} \widehat{a_E}(\p)
&=\sum_{|a|\leq 2\sqrt{q}} a\left(\kappa\frac{H(a,q)}{q^2}\frac{q^{10}}{q^{10}-1}B^{5/6}+O\left(q^{-1}H(a,q)B^{5/6-1/3d}\right)\right)\\
&=0+\sum_{|a|\leq 2\sqrt{q}} O\left(a q^{-1}H(a,q)B^{5/6-1/3d}\right)\\
&\ll \max_{|a|\leq 2\sqrt{q}}\{H(a,q)\} B^{5/6-1/3d}\\
&\ll_{\epsilon} q^{3/2+\epsilon} B^{5/6-1/3d}.
\end{align*}
This, together with the bound (\ref{eq:mult-bound-1}), implies the lemma.
\end{proof}

\begin{lemma}\label{lem:S2}
Let $\p\subset \O_K$ be a prime ideal with norm $N_{K/\Q}(\p)=q$ such that $2\nmid q$ and $3\nmid q$. Then we have the following estimate:
\[
\sum_{E\in \mcE_K(B)} \widehat{a_E}(\p^2)=-\kappa q B^{5/6}+O\left(q^{1/2+\epsilon}B^{\frac{5}{6}}+q^2 B^{\frac{5}{6}-\frac{1}{3d}}\right),
\]
where the implied constant does not depend on $q$.
\end{lemma}

\begin{proof}
Let $\mcE_K^{good}(B)$ (resp.\ $\mcE_K^{\mult}(B)$) denote the set of elliptic curves over $K$ with good reduction (resp.\ multiplicative reduction) at $\p$ and height bounded by $B$. In this case we have that
\[
\sum_{E\in \mcE_K(B)} \widehat{a_E}(\p^2)=\sum_{|a|\leq 2\sqrt{q}} \sum_{\substack{E\in\mcE_K^{good}(B)\\ a_\p(E)=a}} \widehat{a_E}(\p^2) + \sum_{E\in \mcE_K^{\mult}(B)} \widehat{a_E}(\p^2).
\]
Since $\widehat{a_E}(\p^2)=\widehat{a_E}(\p)^2=1$ for all $E\in \mcE_K^{\mult}(B)$, we have, by Theorem \ref{thm:LocalConditions}, that 
\begin{align}\label{eq:mult-bound-2}
\left|\sum_{E\in \mcE_K^{\mult}(B)} \widehat{a_E}(\p^2)\right|=\#\mcE_K^{\mult}(B)=\kappa \frac{q-1}{q^2}\frac{q^{10}}{q^{10}-1} B^{5/6}+O(B^{5/6-1/3d})=O\left(B^{5/6}\right).
\end{align}

In the case that $E$ has good reduction at $\p$, a straightforward computation shows that $\widehat{a_E}(\p^2)$ equals $a_\p(E)^2-2q$. Therefore, by Proposition \ref{prop:ClassNumberSums} and the proof of Theorem \ref{thm:aqLocalCondition}, we have that
\begin{align*}
\sum_{|a|\leq 2\sqrt{q}} \sum_{\substack{E\in\mcE_K^{good}(B)\\ a_\p(E)=a}} &\widehat{a_E}(\p^2)
=\sum_{|a|\leq 2\sqrt{q}} \sum_{\substack{E\in\mcE_K^{good}(B)\\ a_\p(E)=a}} (a_\p(E)^2-2q)\\
&=\sum_{|a|\leq 2\sqrt{q}} (a^2-2q)\left(\kappa\frac{H(a,q)}{q^2}\frac{q^{10}}{q^{10}-1}B^{5/6}+O\left(q^{-1}H(a,q)B^{5/6-1/3d}\right)\right)\\
&=\frac{\kappa q^{8}}{q^{10}-1}\left(\sum_{|a|\leq 2\sqrt{q}} a^2 H(a,q)-2q\sum_{|a|\leq 2\sqrt{q}} H(a,q)\right) B^{\frac{5}{6}}+O\left(q^2 B^{\frac{5}{6}-\frac{1}{3d}}\right)\\
&=\frac{-\kappa q^{11}}{q^{10}-1} B^{5/6}+O\left(q^{1/2+\epsilon}B^{\frac{5}{6}}+q^2 B^{\frac{5}{6}-\frac{1}{3d}}\right).
\end{align*}
Observe that
\[
\frac{q^{10}}{q^{10}-1}=\sum_{k=0}^{\infty} q^{-10k}=1+O(q^{-10}).
\]
Therefore
\[
\sum_{|a|\leq 2\sqrt{q}} \sum_{\substack{E\in\mcE_K^{good}(B)\\ a_\p(E)=a}}\widehat{a_E}(\p^2)=-\kappa q B^{5/6}+O\left(q^{1/2+\epsilon}B^{\frac{5}{6}}+q^2 B^{\frac{5}{6}-\frac{1}{3d}}\right).
\]
This, together with the bound (\ref{eq:mult-bound-2}), implies the lemma.
\end{proof}

\subsection{Bounding the average analytic rank of elliptic curves}

We now prove our main result.

\vspace{2mm} 

\begin{reptheorem}[\ref{thm:AvgRank}]
Let $K$ be a number field of degree $d$. Assume that all elliptic curves over $K$ are modular and that their $L$-functions satisfy the Riemann-Hypothesis. Then the average analytic  rank of isomorphism classes of elliptic curves over $K$, when ordered by naive height, is bounded above by $(9d+1)/2$.
\end{reptheorem}

\begin{proof}
By Lemma \ref{lem:S1} and the observation that $\widehat{\phi}$ is supported on the interval $[-\nu,\nu]$, we have that
\begin{align*}
S_1(\phi,B)&=\frac{2}{\log(B) \#\mcE_{K}(B)} \sum_{\substack{
v\in \Val_0(K)\\ 2\nmid q_v,\ 3\nmid q_v}} \frac{\log(q_v)}{q_v} \cdot \widehat{\phi}\left( \frac{\log(q_v)}{\log(B)}\right)\sum_{E\in\mcE_{K}(B)} \widehat{a_E}(\p_v)\\
&\ll_{\epsilon} \frac{2}{\log(B) \#\mcE_{K}(B)} \sum_{\substack{v\in \Val_0(K)\\ 2\nmid q_v,\ 3\nmid q_v}} \frac{\log(q_v)}{q_v} \cdot \widehat{\phi}\left( \frac{\log(q_v)}{\log(B)}\right)\left(\frac{1}{q_v} B^{5/6}+q_v^{3/2+\epsilon}B^{5/6-1/3d}\right)\\
&\ll \frac{2}{\log(B) \#\mcE_{K}(B)} \sum_{\substack{v\in \Val_0(K) \\ q_v\leq B^\nu \\ 2\nmid q_v,\ 3\nmid q_v}} \frac{\log(q_v)}{q_v}  \left(\frac{1}{q_v} B^{5/6}+q_v^{3/2+\epsilon}B^{5/6-1/3d}\right).
\end{align*}
By Corollary \ref{cor:EllipticCount}, we know that $\#\mcE_K(B)=\kappa B^{5/6}+O(B^{5/6-1/3d})$. Using this and the prime ideal theorem, we have that 
\begin{align}\label{eq:S1}
S_1(\phi,B)
\ll_{\epsilon} \frac{2}{\log(B)} \sum_{\substack{v\in \Val_0(K) \\  q_v\leq B^\nu \\ 2\nmid q_v,\ 3\nmid q_v}} \frac{\log(q_v)}{q_v}  \left(\frac{1}{q_v}+q_v^{3/2+\epsilon}B^{-1/3d}\right) 
\ll\frac{1+B^{(3/2+\epsilon)\nu-1/3d}}{\log(B)}.
\end{align}
By Lemma \ref{lem:S2}, we have 
\begin{align*}
S_2(\phi,B)
&=\frac{2}{\log(B) \#\mcE_{K}(B)} \sum_{\substack{v\in \Val_0(K) \\ 2\nmid q_v,\ 3\nmid q_v}} \frac{\log(q_v)}{q_v^2} \cdot \widehat{\phi}\left(\frac{2\log(q_v)}{\log(B)}\right)\sum_{E\in\mcE_{K}(B)} \widehat{a_E}(\p_v^2)\\
&\hspace{-15mm}=\frac{2}{\log(B) \#\mcE_{K}(B)} \sum_{\substack{v\in \Val_0(K) \\ 2\nmid q_v,\ 3\nmid q_v}} \frac{\log(q_v)}{q_v^2} \cdot \widehat{\phi}\left(\frac{2\log(q_v)}{\log(B)}\right)\left(-\kappa q_v B^{\frac{5}{6}}+O(q_v^{\frac{1}{2}+\epsilon}B^{\frac{5}{6}}+q_v^2 B^{\frac{5}{6}-\frac{1}{3d}})\right).
\end{align*}
Assuming $\epsilon<1/4$ and using $\#\mcE_K(B)=\kappa B^{5/6}+O(B^{5/6-1/3d})$, we have
\begin{align*}
&S_2(\phi,B)
=\frac{2}{\log(B)} \sum_{\substack{v\in \Val_0(K) \\ 2\nmid q_v,\ 3\nmid q_v}} \frac{\log(q_v)}{q_v^2} \cdot \widehat{\phi}\left( \frac{2\log(q_v)}{\log(B)}\right)\left(-q_v +O(q_v^{1/2+\epsilon}+q_v^2 B^{-1/3d})\right)\\
&=\frac{-2}{\log(B)} \sum_{\substack{v\in \Val_0(K) \\ 2\nmid q_v,\ 3\nmid q_v}} \frac{\log(q_v)}{q_v} \cdot \widehat{\phi}\left( \frac{2\log(q_v)}{\log(B)}\right) +O\left(\frac{1}{\log{B}}\sum_{\substack{v\in \Val_0(K)  \\ q_v\leq B^{\nu/2}}}\left(\frac{1}{q_v^{3/2-\epsilon}}+\log(q_v)B^{-1/3d}\right)\right)\\
&=\frac{-1}{\log(B)} \sum_{\substack{ v\in \Val_0(K) \\  q_v\leq B^{\nu/2} \\ 2\nmid q_v,\ 3\nmid q_v}} \left(\frac{\nu \log(q_v)}{2q_v} - \frac{\log(q_v)^2}{q_v\log(B)}\right)
+O\left(\frac{1+B^{\nu/2-1/3d}}{\log(B)}\right).
\end{align*}
Arguments using the prime ideal theorem show that
\[
\sum_{\substack{ v\in \Val_0(K)  \\ q_v\leq X}} \frac{\log(q_v)}{q_v}=\log(X)+O(1)
\]
and
\[
\sum_{\substack{ v\in \Val_0(K)  \\ q_v\leq X}} \frac{\log(q_v)^2}{q_v}=\frac{\log(X)^2}{2}+O(\log\log(X)).
\]
From this it follows that
\begin{align*}
&\sum_{\substack{ v\in \Val_0(K)  \\ q_v\leq B^{\nu/2} \\ 2\nmid q_v,\ 3\nmid q_v}} \left(\frac{\nu \log(q_v)}{2q_v} - \frac{\log(q_v)^2}{q_v\log(B)}\right)\\
&= \frac{\nu}{2} \left(\log(B^{\nu/2})+O(1)\right)-\frac{1}{\log(B)}\left(\frac{\log(B^{\nu/2})^2}{2}+O(\log\log(B^{\nu/2}))\right) \\
&=\left(\frac{\nu^2}{4}\log(B)+O(1)\right)-\left(\frac{\nu^2}{8}\log(B)+O(1)\right)\\
&=\frac{\nu^2}{8}\log(B)+O(1).
\end{align*}
Therefore
\[
S_2(\phi,B)=\frac{-\phi(0)}{2}+O\left(\frac{1+B^{\nu/2-1/3d}}{\log(B)}\right).
\]

Taking $\nu=(3d(3/2+\epsilon))^{-1}$ and combining the above estimate for $S_2$ with our bound (\ref{eq:S1}) for $S_1$, we obtain
\[
-S_1(\phi,B)-S_2(\phi,B)=\frac{\phi(0)}{2}+o(1).
\]
Substituting this into (\ref{eq:explicit-average-rank}) and taking the limit superior as $B$ goes to infinity gives the following upper bound for the average analytic rank of elliptic curves over $K$:
\[
\frac{1}{\nu}+\frac{1}{2}=3d(3/2+\epsilon)+\frac{1}{2},
\]
for any $0<\epsilon<1/4$. Since the limit superior exists, and since $\epsilon$ can be taken arbitrarily small, we obtain the bound $(9d+1)/2$.
\end{proof}

\bibliographystyle{alpha}
\bibliography{bibfile}

\end{document}